%% file: 31.RV2.tex
\documentclass[12pt]{article}
\usepackage{graphicx} 
\usepackage{color}
\usepackage{xspace}
\usepackage{amsmath,amsthm,amsfonts,amssymb,amscd,amsxtra}
\usepackage{multicol}
\usepackage{float}
\usepackage{indentfirst}
\usepackage{comment}
\usepackage{bm}
\usepackage{enumerate}
\usepackage{enumitem}
\usepackage{listings}
\usepackage[margin=1in]{geometry}
\usepackage[toc]{appendix}
\usepackage{ulem}
\input{mymacros}

\newcounter{question}

\newcounter{com}

\newcommand{\norm}[1]{\left\lVert #1 \right\rVert}

\title{The buckling load of cylindrical shells under axial compression depends on the cross-section curvature}
\author{Davit Harutyunyan\thanks{University of California Santa Barbara, harutyunyan@math.ucsb.edu}
 and Andre Martins Rodrigues\thanks{University of California Santa Barbara, andre02@ucsb.edu}}
\date{}

\begin{document}
\maketitle

\begin{abstract}
It is known that the famous theoretical formula by Koiter for the critical buckling load of circular cylindrical shells under axial compression does not 
coincide with the experimental data. Namely, while Koiter's formula predicts linear dependence of the buckling load $\lambda(h)$ of the shell thickness 
$h$ ($h>0$ is a small parameter), one observes the dependence $\lambda(h)\sim h^{3/2}$ in experiments; i.e., the shell buckles at much smaller loads for small thickness. This theoretical prediction failure is believed to be caused by the so-called sensitivity to imperfections phenomenon (both, shape and load). Grabovsky and the first author have rigorously proven in [\textit{J. Nonl. Sci.,}  Vol. 26, Iss. 1, pp. 83--119, Feb. 2016], that in the problem of circular cylindrical shells buckling under axial compression, a small load twist leads to the buckling load scaling $\lambda(h)\sim h^{5/4},$ while shape imperfections are likely to result in the scaling $\lambda(h)\sim h^{3/2}.$ In this work we prove, that in fact the buckling load $\lambda(h)$ of cylindrical (not necessarily circular) shells under vertical compression depends on the curvature of the cross section curve. When the cross section is a convex curve with uniformly positive curvature, 
then $\lambda(h)\sim h,$ and when the the cross section curve has positive curvature except at finitely many points, 
then $C_1h^{8/5}\leq \lambda(h)\leq C_2h^{3/2}$ for $h$ small thickness $h>0.$ 

\end{abstract}

\tableofcontents

\section{Introduction}
\label{sec:1}

Thin-walled shells are, in general, highly efficient structures. In order to produce reliable designs and to avoid unexpected catastrophic failures, one needs to understand buckling. Buckling occurs in a thin structure under loading, when the structure undergoes an overall change in configuration instead of acting in the primary fashion intended by their designers, leading to the failure of the structures. Physically speaking, buckling in a thin shell occurs, when the shell can absorb a great deal of membrane strain energy without deforming too much but it must deform much more in order to absorb an equivalent amount of bending strain energy. When this stored energy is converted into bending energy, buckling occurs, creating a visible change in the geometry of the shell (typically in the form of several dimples) to accommodate all the energy, e.g., Figure 1. Mathematically speaking, buckling can be interpreted as the instability of the equilibrium state that for a certain load will have two possible trajectories to follow, i.e., when in the stress-strain (or stress-deformation) diagram a bifurcation occurs. This phenomenon is mathematically described as the loss of positivity of the second variation of the total energy of the system. In engineering it is essential to have a good estimate on the critical stress that will trigger buckling. In the present work we revisit the problem of buckling of cylindrical shells under axial compression. In that problem, one starts applying homogeneous load of magnitude $\lambda$ to the top of a cylindrical shell that is resting on a substrate, where $\lambda$ is increased continuously from zero. It is observed that at very small load magnitudes $\lambda,$ the cylindrical shell will undergo a homogeneous deformation with no visible geometric changes. Then at some critical value $\lambda=\lambda(h),$ the shell will buckle, producing a variety of deformation patterns, typically in the form of several (or single) dimples 
[\ref{bib:Yoshimura},\ref{bib:Bud.Hut.},\ref{bib:Bushnell},\ref{bib:Lan.Cal.Pal.},\ref{bib:Dog.Kli.Zim.Ode.Ara.},\ref{bib:Hor.Lor.Pel.},\ref{bib:Zhu.Man.Cal.}] shown in Figure 1\footnote{The second and the third cylinders are apparently deeper into the post-buckling regime.} The dimple (dimples) typically appear with a "click" and drop in the load magnitude (which corresponds to the bifurcation point), and disappear when unloading the shell. Some less common buckling patterns, such as formation of waves in the longitudinal direction or periodic-like wrinkling are also possible, see [\ref{bib:Xu.Mic.}] and Figure 14  of [\ref{bib:Xu.Mic.}] for more details.
\begin{figure}
\includegraphics[scale=0.42]{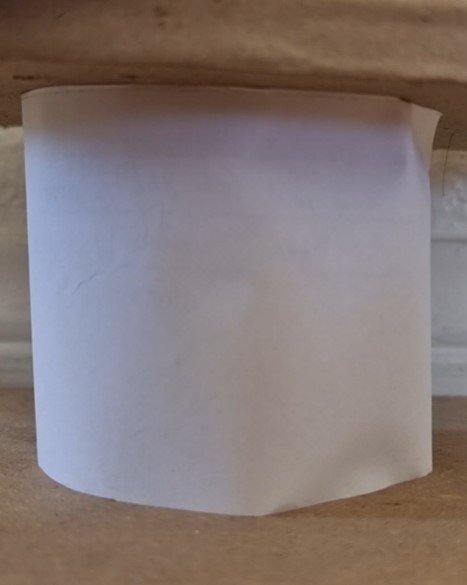}\quad
\includegraphics[scale=0.436]{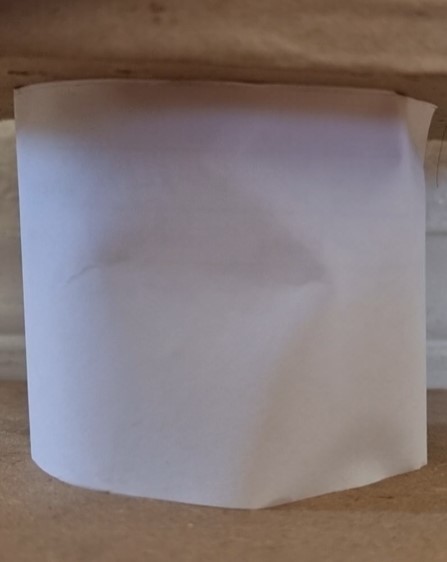}\quad
\includegraphics[scale=0.42]{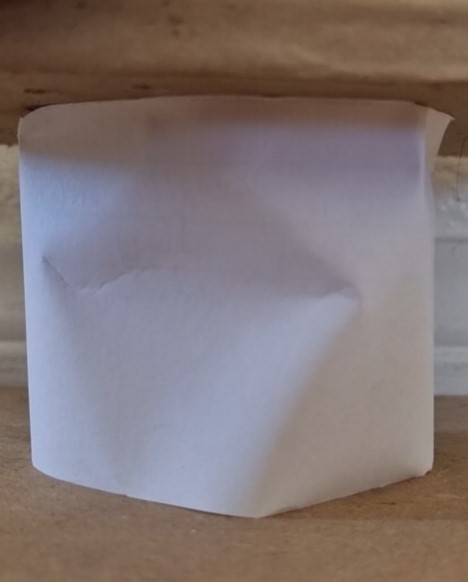}\quad
\caption{Buckled paper cylindrical shells under vertical load.}
\end{figure}

Despite being a very old problem with a lot of data available in the literature, (studies and experiments have been made since the mid-18th century), it still contains several unsolved puzzles even for the simplest geometry of perfectly circular cylindrical shells. For the case of circular cylindrical shells, it has been observed that the buckling load measured by experiments has a large discrepancy with the theoretical predictions made by the classical asymptotic formula 
[\ref{bib:Lorenz},\ref{bib:Timoshenko},\ref{bib:Tim.Woi.},\ref{bib:Koiter}]
\begin{equation}
\label{1.1}
\lambda(h)=\frac{E h}{\sqrt{3\left(1-\nu^{2}\right)}},
\end{equation}
which predicts a linear relation between the thickness of the cylinder and the critical buckling load $\Gl(h).$ Here $E$ and $\nu$ are the Young modulus and the Poisson ratio of the material, respectively, and $h=t/R$ is the shell dimensionless thickness, i.e., the ratio of the cylinder wall thickness $t$ to the cylinder radius $R$.  Note that in fact formula (\ref{1.1}) was first derived by Lorenz [\ref{bib:Lorenz}] in 1911 and independently by Timoshenko [\ref{bib:Timoshenko}] in 1914, but  sometimes in the literature it also carries the name of Koiter, as Koiter derived the so-called Koiter circle associated to (\ref{1.1}) in his Ph.D. thesis [\ref{bib:Koiter}] in 1945, see also [\ref{bib:Gra.Har.3}] for Koiter's circle. On the experimental side, a great deal of experiments since the 1930's show that the experimental critical stress is actually much lower than the theoretical formula (\ref{1.1}), scaling like $h^{3/2}$ with $h$, e.g., [\ref{bib:Lan.Cal.Pal.},\ref{bib:Zhu.Man.Cal.}]. It has been believed in the engineering and applied mathematics communities, that such a paradoxical behavior is in general due to the fact that the buckling load may be highly sensitive to shape or load imperfections [\ref{bib:Almroth},\ref{bib:Tennyson},\ref{bib:Wei.Mor.Sei.},\ref{bib:Gor.Eva.},\ref{bib:Yamaki},\ref{bib:Hun.Lor.Pel.},\ref{bib:Hun.Net.1},\ref{bib:Hun.Net.2},\ref{bib:Lor.Cha.Hun.}]. Note that \textit{general geometric symmetry breaking or (even a very small) preexisting deformation have been believed to be important factors in the asymptotics of the critical buckling load.} These questions have been addressed in the works [\ref{bib:Calladine},\ref{bib:Hor.Lor.Pel.},\ref{bib:Lan.Cal.Pal.},\ref{bib:Zhu.Man.Cal.}] in an attempt to resolve the paradox using mainly numerical approach and/or reduced shell theory equations. One possible gap in these approaches may be whether the utilized reduced shell theory equations are indeed applicable to the problem under consideration and capture the sought parameters within the acceptable error. This concern is based in particular on the works [\ref{bib:Fri.Jam.Mor.Mue.},\ref{bib:Hor.Lew.Pak.},\ref{bib:Lew.Mor.Pak.}], that reveal whether a specific reduced shell theory holds in a specific applied load and elastic energy regime. In the meantime the answer to those specific and important question is unknown. Another weakness was the presence of some heuristic arguments. On the other hand, on the rigorous side, Grabovsky and the first author rigorously proved in [\ref{bib:Gra.Har.3}] that indeed Koiter's asymptotic formula (\ref{1.1}) must hold in the case of perfect cylindrical shells and perfect axial homogeneous loading. This was achieved by the improvement and application of the "Thin structure buckling theory" by Grabovsky and Truskinovsky [\ref{bib:Gra.Tru.}]. Another crucial component of the analysis in [\ref{bib:Gra.Har.3}] was the derivation of the optimal asymptotic constants (not only the asymptotics, but also the leading term in it) in Korn and generalized Korn inequalities for circular cylindrical shells. Then Grabovsky and the first author went on to prove in [\ref{bib:Gra.Har.2}] that in fact if even a very small twist (in the angular direction) is present in the shell loading, then the asymptotics of the buckling load has to drop to $h^{5/4},$ see Table 1 below.

\begin{table}[h]
\centering
\resizebox{\textwidth}{!}{%
\begin{tabular}{|l|c|c|}
\hline
\multicolumn{1}{|c|}{Cross-section (C.-S.) and load} & Circular C.-S., vertical load & Circular C.-S., twist in the load \rule{0pt}{4ex}\\ 
 & $\bm\alpha(\theta)=(\cos(\theta),\sin(\theta))$,  $\Bt=-\lambda\bm e_z$ & $\bm\alpha(\theta)=(\cos(\theta),\sin(\theta))$,  $\Bt=-\lambda(\bm e_z+\epsilon \bm e_\theta)$ \rule{0pt}{2ex}\rule[-2ex]{0pt}{0pt}\\
 \hline
Buckling load asymptotics & $\lambda(h)=\frac{E h}{\sqrt{3\left(1-\nu^{2}\right)}}$ & $\lambda(h)=Ch^{5/4}$ \rule{0pt}{4ex}\rule[-3ex]{0pt}{0pt}\\
\hline
\end{tabular}
}\\
\caption{The dependence of the critical buckling load of circular cylindrical shells on the type of loading. Vertical load versus vertical load with a small twist.}
\end{table}
This analysis to some extent gave an explanation to the fact of sensitivity of the buckling load to load imperfections. Also, a somewhat less rigorous argument in [\ref{bib:Gra.Har.2}] demonstrated why one should expect the buckling load to drop to $h^{3/2}$ in the presence of some small dimples in the shell. 

Our task in the present work is the analysis of the "sensitivity to imperfections" problem for some other kind of shape imperfections, that are diversions from the perfect cylindrical shell. Namely, we will consider cylindrical shells, generated by cylindrical surfaces, that are non-circular but have convex cross sections: Two main families will be analyzed. 
\begin{itemize}
\item[(i)] The first family contains cylindrical surfaces with convex cross sections that have uniformly positive curvature (when regarded from the exterior of the curve). An illustrative example of such a curve is given in the left half of Figure 2. 
\item[(ii)] The second family contains cylindrical suraces with convex cross sections that have uniformly positive curvature (when regarded from the exterior of the curve) except from finitely many points on the curve, where the curvature vanishes. An illustrative example of such a curve is given in the right half of Figure 2. 
\end{itemize}

\begin{figure}
\includegraphics[scale=0.95]{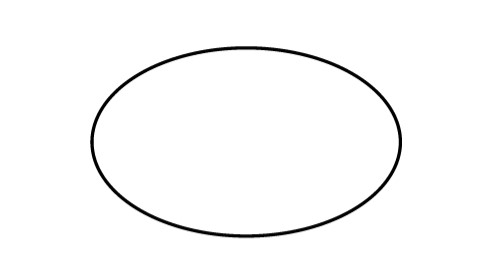}\ \ \ 
\includegraphics[scale=0.95]{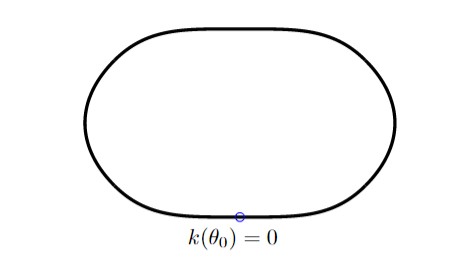}
\caption{Cylindrical surface cross-section curves with uniformly positive curvature (on the left), and uniformly positive curvature, except at finitely many points on the curve, where it vanishes (on the right), e.g., $k(\theta_0)=0.$ The curve on the right behaves like $(t,t^4)$ near the zero curvature points.}
\end{figure}

Note that in both cases the cylinder cross-section does not have to have any geometric symmetry property, but rather only has to be convex with some imposed curvature condition. We will prove that in case (i) the critical buckling load $\Gl(h)$ has the asymptotics $Ch$, and in case (ii) we will prove the bounds
 $C_1h^{8/5}\leq \Gl(h)\leq C_2h^{3/2}$ in the vanishing thickness regime $h\to 0.$ The result in part (i) in particular disproves the longstanding believe that  geometric symmetry breaking would lead to the drop in the asymptotics of $\Gl(h),$ i.e., there is an $\epsilon>0,$ such that $\Gl(h)\leq Ch^{1+\epsilon}$. Also, the result in part (ii) provides new evidence on how a geometric shape imperfection, which is a diversion from the perfect cylindrical shell may lead to the drop in the asymptotics of $\Gl(h)$ to at least $h^{3/2}.$ 

As already pointed out, we will be working in the framework of the (improved) "thin structure buckling theory" of Grabovsky and Truskinovsky [\ref{bib:Gra.Tru.}] rigorously derived from three-dimensional nonlinear hyper-elasticity. Some of the main components in the analysis will be asymptotically sharp Korn 
inequalities for the displacement gradient components, proven for the shells under consideration. 

The paper is organized as follows: In Section~2 we present a brief introduction to the general theory of slender structure buckling [\ref{bib:Gra.Tru.},\ref{bib:Gra.Har.2}]. 
In Section~3 we will be formulating the main results of the paper. In order to apply the theory in Section~2, on one hand one needs to determine 
(with some amount of proximity) the so-called "trivial branch", and the asymptotic stress tensor, which will be done in Section~4.1. Also, on the other hand one needs to prove asymptotically sharp Korn and Korn-type inequalities for the displacement gradient components, thus in Sections~4.2-4.3 we will formulate and prove those. Finally, in Section~4.4 we will prove the main results on the buckling load.

\section{Buckling of slender Structures}
\setcounter{equation}{0}
\label{sec:2}

For the convenience of the reader, this section is devoted to the presentation of the general theory "Buckling of slender structures" by Grabovsky and Truskinovsky [\ref{bib:Gra.Tru.}], that was later extended by Grabovsky and the first author in [\ref{bib:Gra.Har.2}]. In the presentation there will be some little amount of focus on the buckling of cylindrical shells, which is the subject of the paper. While we will try to keep the paper as self contained as possible, we will avoid going into the technical details and proofs referring the reader to the papers [\ref{bib:Gra.Tru.},\ref{bib:Gra.Har.2}] for details. 

\subsection{Loading, deformation, energy}
\label{sec:2.1}

In what follows we will be working in the framework of hyper-elasticity. Let $\Omega\subset\mathbb{R}^3$ be a bounded, simply-connected, and open set with a Lipschitz boundary, representing a hyper-elastic body in space, let $\bm t(\bm x)$ be dead traction applied to the boundary (or just some of it) of $\Omega.$ A resulting deformation $\bm y=\bm y (\bm x),\ (\bm x\in \Omega$) is said to be stable, if it is a weak local minimizer of the total energy of the system, that is given by the formula
\begin{equation}
\label{2.1}
E(\boldsymbol{y})=\int_{\Omega} W(\nabla \boldsymbol{y(\Bx)}) d \Bx -\int_{\partial \Omega} \boldsymbol{y} \cdot \boldsymbol{t}(\boldsymbol{x}) d S(\boldsymbol{x}),
\end{equation}
where $W(\BF)\colon\mathbb R^{3\times 3}\to\mathbb R$ is the elastic energy density function of the body. We will assume that $W$ is of class $C^3$ in some neighborhood of the identity matrix $\BI.$ Also, in hyper-elasticity, $W$ satisfies the following fundamental properties:

\begin{enumerate}[label={(P\arabic*)}]
    \item Positivity: $W(\BF)\geq 0$ for all $\BF\in\mathbb R^{3\times 3},$ and $W(\BI)=0.$
        \item Absence of prestress: $W_{\BF}(\boldsymbol{I})=\mathbf{0}.$ \\
      
        Note that this condition follows from (P1) and the fact that $W$ is $C^3$-regular at $\BI.$ However, this condition is traditionally mentioned as it has the mechanical meaning of the absence of prestress. 
    \item Frame indifference: $W( \boldsymbol{R}\boldsymbol{F})=W(\boldsymbol{F})$ for every $\boldsymbol{R} \in S O(3).$
    \item Local stability of the trivial deformation $\boldsymbol{y}(\boldsymbol{x})=\boldsymbol{x}:\left(\BL_{0} \boldsymbol{\xi}, \boldsymbol{\xi}\right) \geq 0$ for any 
    $\boldsymbol{\xi} \in \mathbb{R}^{3 \times 3},$ where $\BL_{0}=W_{\BF\BF}(\BI)$ is the linearly elastic tensor of material properties.
    \item Non-degeneracy: $\left(\BL_{0} \boldsymbol{\xi}, \boldsymbol{\xi}\right)=0$ if and only if $\boldsymbol{\xi}^{T}=-\boldsymbol{\xi}$.
\end{enumerate}
In addition to the above standard properties (P1)-(P5), we will assume that the material is isotropic, i.e., the energy density $W$ satisfies the additional property (P6) below.
\begin{itemize}
    \item[(P6)] Isotropy: $W(\boldsymbol{F}\BR)=W(\boldsymbol{F})$ for every $\boldsymbol{R} \in S O(3).$    
    \end{itemize}
Here, for a function $W(\BF)\colon\mathbb R^3\to\mathbb R,$ the symbols $W_{\BF}$ and $W_{\BF\BF}$ denote the gradient and the Hessian of $W$ respectively, i.e.,
$$
W_{\BF}(\BF)=\left(\frac{\partial W}{\partial f_{ij}}(\BF)\right)_{1\leq i,j\leq 3} \quad\text{and} \quad
W_{\BF\BF}(\BF)=\left(\frac{\partial^2 W}{\partial f_{ij}\partial f_{kl}}(\BF)\right)_{1\leq i,j,k,l\leq 3},
$$
where $\BF=(f_{ij})_{1\leq i,j\leq 3.}$ As it is known in mechanics, in the presence of isotropy homogeneous deformations are possible. We will also see later in Section~\ref{sec:4.1} how the assumption of isotropy substantially simplifies the job of proving the existence of a trivial branch. Here and in what follows, $(\boldsymbol{A}, \boldsymbol{B})$ is the Frobenius inner product of matrices $\BA=(a_{ij})_{i,j=1}^3,\BB=(b_{ij})_{i,j=1}^3\in\mathbb R^{3\times 3}:$ $(\boldsymbol{A}, \boldsymbol{B})=\sum_{i,j=1}^3 a_{ij}b_{ij}.$ In the theory of hyper-elasticity, one typically chooses for the admissible set $\cal A$ of the deformations $\By(\Bx)$ to be the subspace of functions that belong to some Sobolev space $W^{1,p}(\Omega)$ ($1<p$), which in addition satisfy some Dirichlet or/and natural Neumann boundary conditions on some complementary portions of $\partial\Omega$ yielded from the loading. Then the resulting deformation $\By(\Bx)$ must be a local minimizer of the energy $E(\By)$ in the admissible set $\cal A.$ A more detailed discussion on the deformations of slender structures is presented in the section to follow.

\subsection{Trivial branch}
\label{sec:2.2}

In the general theory of buckling of slender structures originated in [\ref{bib:Gra.Tru.}], one considers a sequence of \textit{slender} domains $\Omega_h$ parametrized by a small parameter $h$. The notion of slenderness used in this paper will be given a precise definition later in Section~2.3. In the case when $\Omega_h$ is a shell, the parameter $h$ typically represents the thickness of the shell, or the dimensionless thickness, i.e., the thickness divided by one of the in-plane parameters. In our case, $h$ will represent the thickness of the non-circular cylinders $\Omega_h$ under consideration, with constant height $L$. Consider a loading program 
\begin{equation}
\label{2.2}
\Bt(\Bx,h,\lambda)=\lambda \Bt^h(\Bx)+O(\lambda^2),\qquad \Bx\in \Gamma_1^h,
\end{equation}
applied to the part $\Gamma_1^h$ of the boundary of $\Omega_h,$ where $\lambda$ is the load magnitude and $\Bt^h$ is the load direction. 
For the problem of cylindrical shell axial compression, the part $\Gamma_1^h$ would be the top and the bottom parts of the boundary only. 
Assume for some $\lambda_0>0$ the loading (\ref{2.2}) results in a family of Lipschitz deformations $\By (\bm x;h,\lambda)\in W^{1,\infty}(\Omega_h,\mathbb R^3)$ for all $\lambda\in[0,\lambda_0],$ where as mentioned above, the field $\bm y (\bm x;h,\lambda)$ is a stable deformation for given boundary conditions, i.e., it is a weak local minimizer of the elastic energy 
$$E_{el.}(\boldsymbol{y})=\int_{\Omega} W(\nabla \boldsymbol{y(\Bx)}) d \Bx$$
in the admissible set 
\begin{equation}
\label{2.3}
\mathcal{A}_\lambda=\{ \By(\Bx,h,\lambda)\in W^{1,\infty}(\Omega_h,\mathbb R^3) \ : 
\ \By(\Bx,\lambda,h)=\Bg(\Bx,\lambda,h)  \ \text{on}\  \ \Gamma_1^h  \ \text{in the trace sence}\},
\end{equation}
for all $\lambda\in [0,\lambda_0],$ where the vector field $\Bg(\Bx,\lambda,h)\in W^{1,\infty}\left(\Omega_{h} ; \mathbb{R}^{3}\right)$ represents the boundary conditions on the part $\Gamma_1^h$ where the load is applied, and it is typically defined specifically for each problem and results from the loading program (\ref{2.2}). The family of deformations $\bm y (\bm x;h,\lambda)$ is then called a \textit{trivial branch}. Next, one defines the so-called \textit{linearly elastic trivial branch.}

\begin{definition}
\label{def:2.1}
A family of stable (within the admissible set $\mathcal{A}_\lambda$) Lipschitz deformations $\By(\Bx,h,\lambda)\in W^{1,\infty}(\Omega_h,\mathbb R^3)$ is called a \textbf{linearly elastic trivial branch,} if there exist $h_{0}>0,$ so that for every $h \in\left[0, h_{0}\right]$ and $\lambda \in\left[0, \lambda_{0}\right]$ one has:
\begin{enumerate}[label=(\roman*)]
\item $\bm y(\bm x; h,0)=\bm x.$ 
\item   There exist a family of Lipschitz functions $\boldsymbol{u}^{h}(\boldsymbol{x})$, independent of $\lambda$, such that
\begin{equation}
\label{2.4}
\left\|\nabla \boldsymbol{y}(\boldsymbol{x} ; h, \lambda)-\boldsymbol{I}-\lambda \nabla \boldsymbol{u}^{h}(\boldsymbol{x})\right\|_{L^{\infty}\left(\Omega_{h}\right)} \leq C \lambda^{2},
\end{equation}
\item 
\begin{equation}
\label{2.5}
\left\|\frac{\partial(\nabla \boldsymbol{y})}{\partial \lambda}(\boldsymbol{x} ; h, \lambda)-\nabla \boldsymbol{u}^{h}(\boldsymbol{x})\right\|_{L^{\infty}\left(\Omega_{h}\right)} \leq C \lambda.
\end{equation}
Here the constant $C>0$ is independent of $h$ and $\lambda$.
\end{enumerate}
\end{definition}
We remark that neither uniqueness nor general stability of the trivial branch are assumed. It is important to note that, here, the term general stability is stability without the boundary conditions in (\ref{2.3}). It is worth mentioning that as understood in [\ref{bib:Gra.Tru.}], usually the family $\By(\Bx,h,\Gl)$ is not stable in general due to the possibility of infinitesimal flips. This is always the case for cylindrical shell compression problems due to possible infinitesimal rotations in the cross-section plane. Additionally the leading order term $\Gl\bm u^h(\bm x)$ of the nonlinear trivial branch is nothing else but the linear elastic displacement, that can be calculated solving the equations of linear elasticity (the Lam\'e system) $\nabla \cdot\left(\BL_{0} e\left(\boldsymbol{u}^{h}\right)\right)=\mathbf{0}$ with the imposed boundary conditions, where $e\left(\Gl\boldsymbol{u}^{h}\right)=\frac{\Gl}{2}(\nabla \boldsymbol{u}^{h}+\left(\nabla \boldsymbol{u}^{h})^{T}\right)$ is the linear strain.

\subsection{The buckling load and buckling modes}

Buckling of the thin structure $\Omega^h$ occurs when the trivial branch $\By(\Bx ; h, \lambda)$ becomes unstable for some value $\lambda=\lambda_{crit}(h)$ for the first time. This happens because it becomes energetically more favorable to bend the structure rather than store more compressive energy. This bifurcation is detected by the change in the sign of the second variation of the energy:
\begin{equation}
\label{2.6}
\delta^{2}E(\BGf ; h, \lambda)=\int_{\Omega_{h}}\left(W_{F F}(\nabla \By(x ; h, \lambda)) \nabla \BGf, \nabla \BGf \right) d x, \qquad \BGf \in V_h,
\end{equation}
where $V_h\subset W^{1,\infty}(\Omega^h)$ is the vector space of admissible variations resulting from the loading program (\ref{2.2}), i.e., from the admissible set 
$\mathcal{A}_\lambda$ in (\ref{2.3}). Namely, this means that there exists $\lambda_{crit}(h)>0$ such that the second variation is non negative when $0<\lambda<\lambda_{crit}(h)$ for all test functions $\BGf  \in V_h,$ but for 
$\lambda>\lambda_{crit}(h)$ it can become negative for some choice of $\BGf.$ This observation leads to the following mathematical definition of the critical buckling load:
\begin{equation}
\label{2.7}
\lambda_{\text {crit }}(h)=\inf \left\{\lambda>0: \delta^{2} E(\BGf ; h, \lambda)<0 \text { for some } \BGf \in V_h\right\}.
\end{equation}
The body $\Omega_h$ is said to undergo a \textit{near-flip buckling}, if $\lim_{h\to 0} \lambda_{\text {crit }}(h)=0.$ A buckling mode is generally understood as a different from zero variation $\BGf^*_h\in V_h$, such that $\delta^{2} E\left(\BGf_{h}^{*}; h, \lambda_{crit}(h)\right)=0$. In practice one is only interested in the leading term asymptotics (or sometimes even in the scaling) of the buckling load in $h$ as $h\to 0.$ Note that if one perturbs $\Gl_{\text {crit }}(h)$ by a small factor $\epsilon,$ i.e., one replaces it by $(1+\epsilon)\Gl_{\text {crit }}(h),$ then the second variation changes approximately by $\epsilon\Gl_{\text {crit }}(h) \frac{\partial\left(\delta^{2} E\right)}{\partial \lambda}\left(\BGf_{h}^* ; h, \lambda_{\text {crit }}(h)\right)$. This means that second variations at two different variations $\BGf_1$ and $\BGf_2$ differing by an infinitesimal value compared to  
$\Gl_{\text {crit }}(h) \frac{\partial\left(\delta^{2} E\right)}{\partial \lambda}\left(\BGf_{h}^* ; h, \lambda_{\text {crit }}(h)\right)$ should not be distinguished. This observation leads to the following new definition of buckling loads and buckling modes in the broader sense. 
\begin{definition}
\label{def:2.2} 
We say that a function $\lambda(h) \rightarrow 0$, as $h \rightarrow 0$ is a buckling load if
$$
\lim _{h \rightarrow 0} \frac{\Gl(h)}{\Gl_{\text {crit }}(h)}=1.
$$
Similarly, a buckling mode is a family of variations $\BGf_{h} \in V_h \backslash\{0\},$ such that
$$
\lim _{h \rightarrow 0} \frac{\delta^{2} E\left(\BGf_{h} ; h, \Gl_{\text {crit }}(h)\right)}{\Gl_{\text {crit }}(h)\frac{\partial\left(\delta^{2} E\right)}{\partial \Gl}\left(\BGf_{h} ; h, \Gl_{\text {crit }}(h)\right)}=0
$$
\end{definition}
It turns out that under the conditions in Definition~\ref{def:2.1} [\ref{bib:Gra.Tru.}], the buckling load and buckling modes can be captured by the so called \textit{constitutively linearized second variation:}
\begin{equation}
\label{2.8}
\delta^{2} E_{\mathrm{cl}}(\BGf ; h, \Gl)=\int_{\Omega_h}\left(\left \langle\BL_{0} e(\BGf), e(\BGf)\right\rangle+\lambda\left\langle\BGs_{h}, \nabla \BGf^{T} \nabla \BGf\right\rangle\right) d\Bx, \qquad \BGf \in V_h,
\end{equation}
where $\BGs_{h}(\boldsymbol{x})=\BL_{0} e\left(\boldsymbol{u}^{h}(\boldsymbol{x})\right)$ is the linear elastic stress.
From the condition (P4) on $\BL_0,$ the first term in (\ref{2.8}) is always nonnegative, thus the potentially destabilizing variations will be in the set: 
\begin{equation}
\label{2.9}
V_h^d=\left\{\phi \in V_h :\left\langle\BGs_h, \nabla \BGf^{T} \nabla \BGf\right\rangle\leq 0\right\}.
\end{equation}
Further, the \textit{constitutively linearized critical strain} will be obtained by minimizing the Rayleigh quotient
\begin{equation}
\label{2.10}
\Re(h, \BGf)=-\frac{\int_{\Omega_{h}}\left\langle \BL_{0} e(\BGf), e(\BGf)\right\rangle d x}{\int_{\Omega_{h}}\left\langle\BGs_{h}, \nabla \BGf^{T} \nabla\BGf \right\rangle d x},
\end{equation}
over the set of destabilizing variations $V_h^d.$ One can formalize this in the following definition.
\begin{definition}
\label{def:2.3}
The constitutively linearized buckling load $\lambda_{\mathrm{cl}}(h)$ is defined by
\begin{equation}
\label{2.11}
\lambda_{\mathrm{cl}}(h)=\inf _{\BGf \in V_h^d} \Re(h, \BGf).
\end{equation}
The family of variations $\left\{\BGf_{h} \in V_h^d: h \in\left(0, h_{0}\right)\right\}$ is called a constitutively linearized buckling mode if
\begin{equation}
\label{2.12}
\lim _{h \rightarrow 0} \frac{\Re\left(h, \BGf_{h}\right)}{\Gl_{\mathrm{cl}}(h)}=1.
\end{equation}
\end{definition}

\begin{definition}
\label{def:2.4}
The body $\Omega_{h}$ is slender if $\lim _{h \rightarrow 0} K\left(\Omega_{h}\right)=0,$ where the Korn constant is defined as
$$
K\left(\Omega_h \right)=\inf _{\phi \in V_h} \frac{\|e(\BGf)\|_{L^{2}\left(\Omega_{h}\right)}^{2}}{\|\nabla \BGf\|_{L^{2}\left(\Omega_{h}\right)}^{2}}.
$$
\end{definition}

The following theorem proven in [\ref{bib:Gra.Har.2}] provides a formula for the buckling load and buckling modes. 

\begin{theorem}
\label{th:2.5}
Suppose that the body is slender in the sense of Definition~\ref{def:2.4}. Assume that the constitutively linearized critical load $\lambda_{\mathrm{cl}}(h),$ defined in (\ref{2.11}) satisfies $\lambda_{\mathrm{cl}}(h)>0$ for all sufficiently small $h$ and
$$
\lim _{h \rightarrow 0} \frac{\lambda_{\mathrm{cl}}(h)^{2}}{K\left(\Omega_{h}\right)}=0.
$$
Then $\lambda_{\mathrm{cl}}(h)$ is the buckling load and any constitutively linearized buckling mode $\BGf_{h}$ is a buckling mode in the sense of Definition \ref{def:2.2}.
\end{theorem}

Now, in order to study the problem of buckling of cylindrical shells under axial compression, we see from Theorem~\ref{th:2.5} and from the above formalism, that we need to gain appropriate information about the trivial branch associated to the axial compression, the associated stress tensor, the Korn constant, and the load 
$\lambda_{\mathrm{cl}}(h).$ This is done in the next sections.

\section{Problem setting and main results}
\setcounter{equation}{0}
\label{sec:3}

In this section we formulate the problems under consideration and the main results. Let $\bm\alpha(\theta)$ be a closed convex curve in the $XY$-plane that is parameterized by the arc length and has period/length $p$. Then the cylindrical surface $S$ having horizontal cross section as the curve $\BGa$ and height $L>0$ will be given by the formula 
$$S: \ \ \bm{r}(\theta,z)=\bm\alpha(\theta)+z\bm{e}_z,\quad  \theta\in[0,p],\  z\in[0,L].$$
Denote $\Be_\Gth=\BGa'(\Gth)$ and by $\Be_t$ the tangent and normal vectors to $\BGa$ respectively. This will give rise to the local orthonormal basis 
$(\Be_t,\Be_\Gth,\Be_z).$ We will be dealing with cylindrical shells $\Omega_h$ with mid-surface $S$ and thickness $h,$ i.e., 
\begin{equation}
\label{3.1}
\Omega_{h}=\left\{(t, \theta, z) \ : \ t \in [-h/2,h/2], \ \theta \in[0, p], \ z \in[0, L]\right\}.
\end{equation}
Denoting $k_\theta=\|\bm\alpha''\|$ the curvature of the cross-section $\BGa,$ we have that the gradient of any vector field 
$\BGf=(\phi_t,\phi_\Gth,\phi_z)\in H^1(\Omega_h,\mathbb R^3)$ will be given in the local basis $(\bm{e}_t,\bm{e}_\theta,\bm{e}_z)$ by 
\begin{equation}
\label{3.2}
\nabla \BGf
=\begin{bmatrix}
\phi_{t,t} & \frac{\phi_{t,\theta}-k_{\theta }\,\phi_\theta}{k_{\theta }\,t+1} & \phi_{t,z} \\[10pt]
\phi_{\theta,t} & \frac{\phi_{\theta,\theta}+k_{\theta }\,\phi_t}{k_{\theta }\,t+1} & \phi_{\theta,z}\\[10pt] 
\phi_{z,t} & \frac{\phi_{z,\theta}}{k_{\theta }\,t+1} & \phi_{z,z} 
\end{bmatrix},
\end{equation}
where $f_{,\eta}$ inside the gradient matrix denotes the partial derivative $\partial_\eta f.$ Assume that the shell $\Omega_h$ is resting on a substrate and is undergoing uniform and homogeneous vertical/axial loading\footnote{The terminology "axial load" will be used even if $\BGa$ does not have point-symmetry, and thus $\Omega_h$ will have no axis.} 
\begin{equation}
\label{3.3}
\Bt(\Bx,h,\lambda)=-\lambda \Be_z
\end{equation}
applied at the top of the shell. The main problem that we are concerned with is the determination of the asymptotics of critical buckling load $\Gl(h)$ in the thickness $h$ as $h\to 0.$ As already mentioned in Section~\ref{sec:1}, one has $\Gl(h)\sim h$ as $h\to 0$ in the case when the cross section $\BGa$ is a circle. We will be considering the cases when $\BGa$ has positive curvature everywhere ($k(\Gth)>0$ for all $\Gth\in [0,p]$), and the case when $\BGa$ has positive curvature 
everywhere except for finitely many points on the curve, at which the curvature has to vanish as $\BGa$ is convex and thus $k(\Gth)\geq 0$ for all $\Gth\in [0,p].$ 
Let us mention that in what follows the norm $\|f\|$ will be the $L^2$ norm $\|f\|_{L^2(\Omega_h)}.$ For the problem under consideration, the natural choice for the vector space $V_h$ will be the subspace of all displacements $\BGf\in H^1(\Omega_h)$ that vanish at the top and the bottom of the shell, i.e,  
\begin{equation}
\label{3.4}
V_h=\{\BGf\in H^1(\Omega_h) \ : \ \BGf(t,\Gth,0)=\BGf(t,\Gth,L)=0, \ \ (t,\Gth)\in [-h/2,h/2]\times [0,p] \}.
\end{equation}
This means that one shall study the stability of the trivial branch within the set of all Lipschitz deformations satisfying the same Dirichlet boundary conditions on the top and bottom of the cylinder. However, we will prove a stronger stability result, i.e., stability within a wider class of deformations. Note that if one only prescribes the values of the vertical component $y_z$ of the trivial branch $\By$ at the top and the bottom of the shell, then the cylinder may undergo flip instability through infinitesimal rotations in the cross section plane [\ref{bib:Gra.Har.2},\ref{bib:Gra.Tru.}]. Therefore one needs to impose some condition on $y_t$ or $y_\theta$ to rule this possibility out. Following [\ref{bib:Gra.Har.3}], we choose to impose zero integral condition on the tangential component $\phi_\Gth$ in the $z$ direction, which gives the alternative subspace
\begin{equation}
\label{3.4.1}
V_h^\Gth=\{\BGf\in H^1(\Omega_h) \ : \  \phi_z|_{z=0}=\phi_z|_{z=L}=\int_0^L\phi_\Gth(t,\Gth,z)dz=0, \ \ \forall (t,\Gth)\in [-h/2,h/2]\times [0,p] \}.
\end{equation}
Let us explain this choice. Observe that we can extend the cylinder to the lower half-space by mirror reflection about the plane $z=0,$ and consequently the components $\phi_t$ and $\phi_\Gth$ as even functions and $\phi_z$ as an odd function, preserving the structure of the symmetric gradient $e(\BGf).$ This would allow us to expand the components $\phi_t$ and $\phi_\Gth$ in Fourier space in cosine series and the $\phi_z$ component in sine series in the $z$ variable; see (\ref{4.19}) and the paragraph between (\ref{4.18}) and (\ref{4.19}). Then the zero integral condition simply means that the independent of $z$ term in the expansion of $\phi_\Gth$ is nonexistent. This would clearly prevent the cylinder from undergoing infinitesimal rotations in the cross section plane, which affects only the independent of $z$ variable terms in $\BGf.$ Note also, that unlike perfect circular cylindrical shells, where homogeneous deformations result in no change in the tangential component, (thus in [\ref{bib:Gra.Har.2}], the authors have prescribed the tangential component $\phi_\Gth$ at the top and the bottom of the cylinder too), homogeneous deformations in general cylindrical shells with any cross sections result in nonzero displacements in the tangential component too, thus a more relaxed subspace, such as (\ref{3.4.1}) has to be considered. We will prove the following theorems, where the vector space $V$ is either $V_h$ or $V_h^\Gth.$

\begin{theorem}
\label{th:3.1}
Assume $\BGa$ is a convex $C^2$ regular curve in the $XY$ plane and assume the cylindrical shell $\Omega_h$ is given as in (\ref{3.1}). Assume $\Omega_h$ is undergoing vertical loading as in (\ref{3.3}), and the admissible variations $\BGf$ belong to the subspace $V.$ The following statements hold:
\begin{itemize}
\item[(i)] If 
$$\min_{\Gth\in[0,p]}k_\Gth(\Gth)=k>0,$$
then one has 
$$\Gl(h)\sim h,$$
as $h\to 0.$
\item[(ii)]  Assume in addition that the cross section $\BGa$ is of class\footnote{Note that the higher regularity $C^5$ is only required for the Ansatz construction in (\ref{4.50}).} $C^5,$ and assume $k_\Gth(\Gth)>0$ except for finitely many points $\Gth_1,\Gth_2,\dots,\Gth_n\in (0,p),$ where one has $k_\Gth(\Gth_i)=0$ and $k_\Gth''(\Gth_i)\neq 0$\footnote{Due to the fact $k_\Gth\geq 0,$ we have $k_\Gth'(\Gth_i)=0$ for $i=1,2,\dots,n.$} for $i=1,2,\dots,n.$ This will imply quadratic growth of the curvature at the points $\Gth_i,$ i.e., 
\begin{equation}
\label{3.6}
c|\Gth-\Gth_i|^2\leq k_\Gth(\Gth)\leq \frac{1}{c}|\Gth-\Gth_i|^2,\qquad \Gth\in [0,p],\ i=1,2,\dots,n,
\end{equation}
for some constant $c>0.$ Then there exist constants $C_1,C_2>0,$ depending only on the cylinder height $L$ and the cross-section 
$\BGa,$ such that one has 
$$C_1h^{8/5}\leq \Gl(h)\leq C_2h^{3/2},$$
 as $h\to 0.$
\end{itemize}
\end{theorem}

A remark is in order. 

\begin{remark}
\label{rem:3.2}
Part (i) of Thereom~\ref{th:3.1} implies that no matter what the geometry of the cross section $\BGa$ is, the critical buckling load $\Gl(h)$ will scale like $h$ as $h\to 0$
as long as $\BGa$ has strictly positive curvature everywhere. This in particular means that the buckling load asymptotics in the vanishing thickness is independent of the fact whether $\BGa$ has any kind of symmetry or not. Hence, \textbf{the critical buckling load asymptotics is not sensitive
to the initial symmetry of the undeformed configuration.} Part (ii) provides some evidence for the phenomenon that the buckling load may drop to $h^{3/2}$ if the cylinder has some zero longitudinal curvature sections. 
\end{remark}

In order to prove Theorem~\ref{th:3.1}, we will need the following Korn and Korn-like inequalities that can be considered as part of the main results of the paper. 

\begin{theorem}
\label{th:3.3}
Let the cylindrical shell $\Omega_h$ be as in Theorem~\ref{th:3.1}. 
The following statements hold:
\begin{itemize}
\item[(i)] If 
$$\min_{\Gth\in[0,p]}k_\Gth(\Gth)=k>0$$
then there exist a constant $\tilde h>0,$ depending only on $L, \max k_\Gth,$ and $k,$ such that
\begin{equation}
\label{3.5}
\inf_{\Bu\in V}\frac{\|e(\BGf)\|^2}{\|\mathrm{col}_3(\nabla\BGf)\|^2}\sim h,
\end{equation}
for all $h\in (0,\tilde h).$
\item[(ii)] Assume the conditions in part (ii) of Theorem~\ref{th:3.1} are satisfied. Then there exist constants $C_1,C_2,C_3,C_4,\tilde h>0,$ depending only on 
$L,\max k_\Gth,$ and $c,$ such that the Korn and Korn-like inequalities hold:
\begin{equation}
\label{3.7}
C_1h^{12/7}\leq \inf_{\BGf\in V}\frac{\|e(\BGf)\|^2}{\|\nabla \BGf\|^2}\leq C_2h^{5/3},\quad 
\quad  C_3h^{8/5}\leq \inf_{\BGf\in V}\frac{\|e(\BGf)\|^2}{\|\mathrm{col}_3(\nabla\BGf)\|^2}\leq C_4h^{3/2},
\end{equation}
for all $h\in (0,\tilde h).$ These results are presented in the schematic table 2 below.
\end{itemize}

\end{theorem}

\begin{remark}
\label{rem:3.4}
It is proven in [\ref{bib:Gra.Har.4}, Theorem~3.2] that under the assumptions in part (i) of Theorem~\ref{th:3.3}, Korn's first inequality holds:
\begin{equation}
\label{3.8}
\|\nabla\BGf\|^2\leq\frac{C}{h^{3/2}}\|e(\BGf)\|^2,
\end{equation}
for all vector fields $\BGf\in V_h$ and all $h\in (0,\tilde h).$ This will be a useful component in the proof of part (i) of Theorem~\ref{th:3.1}.
\end{remark}



\begin{table}[h]
\centering
\resizebox{\textwidth}{!}{%
\begin{tabular}{|l|c|c|}
\hline
\multicolumn{1}{|c|}{Cross-section (C.-S.) and load} & C.-S. with uniformly positive curvature, vertical load & C.-S. with positive curvature, except at finitely many points, vertical load \rule{0pt}{4ex}\rule[-1.2ex]{0pt}{0pt} \\
 & $k(\theta)>0,$ $\Bt=-\lambda\bm e_z$ & $k(\theta)\geq 0$ and $k(\theta_n)= 0 $ 
 for finitely many $\theta_n$, $\Bt=
 -\lambda\bm e_z$ \rule{0pt}{1ex}\rule[-3ex]{0pt}{0pt}\\ 
 \hline
\multicolumn{1}{|c|}{Buckling load asymptotics}  & $\lambda(h)=Ch$ &$Ch^{8/5}\leq\lambda(h)\leq Ch^{3/2}$ \rule{0pt}{4ex}\rule[-3ex]{0pt}{0pt}\\ \hline
\end{tabular}%
}\\
\caption{The dependence of the critical buckling load of convex cylindrical shells on the cross-section curvature. Cross-sections with uniformly positive curvature versus cross-sections with nonnegative curvature, that vanishes only at finitely many points on the curve.}
\end{table}



\section{Proof of the main results}
\label{sec:4}



\subsection{The trivial branch and the stress tensor}
\setcounter{equation}{0}
\label{sec:4.1}

In this section we will demonstrate how one can apply the theory presented in Section~\ref{sec:2} to the problem of axial compression of cylindrical shells with any convex cross-section. First we will calculate the trivial branch $\By(\Bx,\lambda,h)$ resulting from the compression $\Bt(\Bx,\lambda,h)=-\lambda \Be_z$
at the top of the shell. The vector field $\Bu^h$ then will be obtained, which will yield a formula for the linear elastic stress tensor $\BGs_h$, thus the minimization problem (\ref{2.11}) will be identified, which will be studied in the next section. The compression problem is equivalent to the boundary-value problem, where one prescribes the displacement at the top and the bottom of the shell and at the same time solving the system of Euler-Lagrange equations (equations of equilibrium). The boundary conditions on the vertical component of $\By$ yielding from the axial compression will be $y_z(t,\Gth,0,\Gl,h)=0$ at the bottom and $y_z(t,\Gth,L,\Gl,h)=(1-\lambda)L$ at the top. It is natural to seek the trivial branch among homogeneous deformations, yielding homogeneous thickening of the shell in the cross-section plane. Those are clearly given by 
\begin{equation}
\label{4.1}
y_t=(1+a(\lambda))(t+\bm\alpha(\theta)\cdot \bm e_t(\theta)), \qquad y_\theta=(1+a(\lambda))\bm\alpha(\theta)\cdot \bm e_\theta(\theta), \qquad y_z=(1-\lambda)z, 
\end{equation}
where $\bm y(t,\theta,z)=y_t \bm e_t+y_\theta \bm e_\theta+y_z\bm e_z,$ and $a(\lambda)$ is a smooth enough function satisfying $a(0)=0,$ and will be determined by the equations of equilibrium and the natural boundary conditions yielding from the energy minimization problem. For the sake of notation simplicity, we may abbreviate $\bm y(t,\theta,z;h,\lambda)=\bm y(t,\theta,z)=\bm y(\bm x)$ while keeping in mind that the trivial branch can depend on $h$ and $\lambda$. The plan is to prove the existence of a trivial branch given as in (\ref{4.1}). It is easy to see that when minimizing the elastic energy $\int_{\Omega_h} W(\nabla \bar \By)d\Bx$ subject to boundary conditions $\bar \By(\Bx)$=$\By(\Bx)$ at $z=0,L$ for $\By$ as in (\ref{4.1}), then any local minimizer $\bar \By\in H^1(\Omega_h)$ must satisfy the equations of equilibrium 
\begin{equation}
\label{4.2}
\nabla\cdot \bm P(\nabla \bar\By(\Bx))=0, \qquad \Bx\in \Omega_h,
\end{equation}
together with the natural boundary conditions 
\begin{equation}
\label{4.3}
\bm P(\nabla \bar\By) \Be_t=0 \qquad\text{at}\qquad t=\pm h/2,
\end{equation}
where $\boldsymbol{P}(\boldsymbol{F})=W_{\boldsymbol{F}}(\boldsymbol{F})$  is the Piola-Kirchhoff stress tensor. Also, note that if we do not prescribe the $t$ or $\Gth$ component of the field $\bar y$ at $z=0,L,$ then we will get additional natural boundary conditions such as 
\begin{align}
\label{4.4}
&\bm P (\nabla\bar\By)\Be_z\cdot \Be_t=0    \qquad\text{at}\qquad z=0,L,\\ \nonumber
&\bm P (\nabla\bar\By)\Be_z\cdot \Be_\Gth=0    \qquad\text{at}\qquad z=0,L,
\end{align}
respectively. The existence of the trivial branch is proved in the following Lemma:

\begin{lemma}
\label{lem:4.1}
Assume that $W(\bm F)$ is three times continuously differentiable in a neighborhood of the identity matrix $\bm F=\BI$ and satisfies the properties (P1)-(P6). Then there exists a constant $\Gl_0>0$ and a unique function $a(\lambda)\in C^2([0,\Gl_0],\mathbb R)$, such that $a(0)=0$ and the family $\By(\Bx)=(y_t,y_\Gth,y_z)$ 
given by (\ref{4.1}) satisfies the equations of equilibrium (\ref{4.2}) and all of the boundary conditions (\ref{4.3})-(\ref{4.4}). Moreover, the trivial branch $\By(\Bx)$ also fulfills all the conditions in Definition~\ref{def:2.1} as is required for the general theory to apply. 

\end{lemma}

\begin{proof}
For the proof we adopt the strategy in [\ref{bib:Gra.Har.2}]. First of all note that as homogeneous deformations always satisfy the equations of equilibrium, we only have to verify the boundary conditions in (\ref{4.3}) and (\ref{4.4}). Letting $\bm F=\nabla\bm y$ and $\bm C= \bm F^T\bm F$, we have by the isotropy (P6) that $W(\bm F)=\tilde W\left(\bm C\right)$ for some function $\tilde W$ that is three times continuously differentiable in a neighborhood of the identity matrix $\boldsymbol{I}$. We have by simple algebra for the Piola-Kirchhoff stress tensor the formula
\begin{equation}
\label{4.5}
\boldsymbol{P}(\boldsymbol{F})=W_{\boldsymbol{F}}(\boldsymbol{F})=2 \boldsymbol{F} \tilde{W}_{\BC}\left(\boldsymbol{C}\right).
\end{equation}
Taking into account the form given in (\ref{4.1}) and Frenet?Serret formulas $\bm\alpha''=-k_\theta \Be_t,$ $\bm e'_t=k_\theta\BGa',$ we have 
\begin{align*}
    y_{t,t}&=1+a(\lambda),\quad y_{t,\theta}=k_\theta(1+a(\lambda))\bm\alpha\cdot \bm \alpha', \quad y_{t,z}=0,\\
    y_{\Gth,t}&=0,\quad y_{\theta,\theta}=(1+a(\lambda))(1+\bm\alpha\cdot \bm\alpha''),\quad y_{\Gth,z}=0,\\
    y_{z,t}&=0,\quad y_{z,\Gth}=0,\quad y_{z,z}=1-\lambda.
\end{align*}
Consequently, recalling the formula (\ref{3.2}), we have that the gradient $\nabla \By$ in the curvilinear coordinates $\bm{e}_t,\bm{e}_\theta,\bm{e}_z$ is represented as:
\begin{equation}
\label{4.6}
\nabla \By=\begin{bmatrix} y_{t,t} & \frac{y_{t,\theta}-k_{\theta }\,y_\theta}{k_{\theta }\,t+1} & y_{t,z} \\[10pt]
y_{\theta,t} & \frac{y_{\theta,\theta}+k_{\theta }\,y_t}{k_{\theta }\,t+1} & y_{\theta,z}\\[10pt]
y_{z,t} & \frac{y_{z,\theta}}{k_{\theta }\,t+1} & y_{z,z} \end{bmatrix}=
 \begin{bmatrix}
1+a(\Gl) & 0 &0 \\[10pt]
 0 & 1+a(\Gl) & 0 \\[10pt]
  0 & 0 & 1-\Gl
\end{bmatrix},
\end{equation}
and thus we can calculate
\begin{equation}
\label{4.7}
\BC=\begin{bmatrix}
(1+a(\Gl))^2 & 0 &0 \\[10pt]
 0 & (1+a(\Gl))^2  & 0 \\[10pt]
 0 & 0 & (1-\Gl)^2
\end{bmatrix}.
\end{equation}
This clearly shows that the matrices $\BF$ and $\BC=\BF^T \BF$ are diagonal. Additionally, the frame indifference condition (P3) implies that $\tilde{W}\left(\bm R \bm C \bm R^T\right)=\tilde W(\bm C)$ for all $\bm R \in S O(3)$. Differentiating this relation in $\bm R$ at $\boldsymbol{R}=\boldsymbol{I}$ one concludes that 
$\tilde{W}_{\boldsymbol{C}}(\boldsymbol{C})$ commutes with $\boldsymbol{C}$ for any $\BF\in\mathbb R^{3\times 3}$ and $\BC=\BF^T \BF;$ see [\ref{bib:Gur.Fri.Ana.}, page 290] for details. Consequently if $\BC$ is diagonal with distinct diagonal entries, then $\tilde{W}_{\boldsymbol{C}}(\boldsymbol{C})$ must be diagonal as well, while if $\BC$ is diagonal with possibly equal diagonal entries, then small diagonal perturbations will reduce the situation to the first case, and thus  
$\tilde{W_\BC}(\BC)$ will be diagonal too as long as $\BC$ is so. Hence, as in our case the matrix $\BC$ is diagonal, also $\tilde{W_\BC}(\BC)$ and the Piola-Kirchhoff stress tensor $\bm P(\bm F)$ given in (\ref{4.5}) have to be diagonal as well. 
Therefore the boundary conditions in (\ref{4.4}) are automatically satisfied, and the one in (\ref{4.3}) reduces to the single equation
\begin{equation}
\label{4.8}
\hat{W}_{\BC}\left((1+a(\lambda))^{2}\left(\Be_{t} \otimes \Be_{t}+\Be_{\theta} \otimes \Be_{\theta}\right)+(1-\lambda)^{2} \Be_{z} \otimes \Be_{z}\right)
 \Be_{t} \cdot \Be_{t}=0, 
\end{equation}
the solvability of which is basically guaranteed by the implicit function theorem. Indeed, the absence of prestress condition (P2) and (\ref{4.5}) for $\BF=\BI$ imply that equality (\ref{4.8}) is fulfilled at the point $(0,a(0))=(0,0).$ For the nonzero condition in the implicit function theorem at $\lambda=0$ we get 
$$\BL_0\bm e_t\cdot \bm e_t\neq 0,$$
which is fulfilled due to (P5). Consequently, the implicit function theorem guarantees the existence of a $C^2$ smooth $a(\lambda)$ function in some neighborhood of 
$\lambda=0,$ which completes the proof of the existence part of Lemma~\ref{lem:4.1}. Finally we calculate
$$\Bu^h(\Bx)=\frac{\partial \By(\Bx;h,\lambda)}{\partial \lambda}\bigg\vert_{\lambda=0}=
a'(0)(t+\BGa\cdot\Be_t)\bm e_t+(a'(0)\BGa\cdot\Be_\Gth) \bm e_\Gth-z\bm e_z,$$
which does not depend on $h$ and thus by the $C^2$ smoothness of the function $a(\lambda)$ we conclude that the conditions (\ref{2.4}) and (\ref{2.5})
are fulfilled as well. 

\end{proof}

Now in order to identify the minimization problem (\ref{2.11}), we need to calculate the linear elastic stress tensor. We have differentiating (\ref{4.6}) in $\lambda$ at $\lambda=0,$ that 
$$
\nabla \Bu^h(x)=
\begin{bmatrix} 
a'(0) & 0 & 0\\[10pt]
0 & a'(0) & 0\\[10pt] 
0&0&-1
\end{bmatrix},
$$
thus we get 
$$e(\Bu^h(x))=
\begin{bmatrix} 
a'(0) &  0 & 0\\[10pt]
 0 & a'(0)& 0\\[10pt] 
 0 &0 &-1
 \end{bmatrix},
 $$
 where $a'(0)=\nu$ is the Poisson's ratio. As the material is isotropic we have for the 
 linear elastic stress tensor $\BGs_h$ the formula $\Gs_h^{ij}=\frac{2\mu\nu}{1-2\nu}\delta_{ij} \mathrm{Tr}(e(\Bu^h))+2\mu e_{ij}(\Bu^h),$ where $\mu$ is the Lam\'e parameter. This leads to the form 
\begin{equation}
\label{4.9} 
\BGs_h=
 \begin{bmatrix} 
0 & 0& 0\\[10pt]
 0 & 0 & 0\\[10pt]
 0 & 0 & -E
 \end{bmatrix},
 \end{equation}
where $E$ is the Young's modulus. Now we can identify the set $V_h^d$ of destabilizing variations given in (\ref{2.9}), which will be the set all variations $\BGf\in V_h$ such that 
$$-\int_{\Omega^h}(\BGs_h,\nabla\BGf^T\nabla\BGf)d\Bx=E(\left\|\phi_{t,z}\right\|^{2}+\left\|\phi_{\theta,z}\right\|^{2}+\left\|\phi_{z,z}\right\|^{2})\geq 0,$$
i.e., it coincides exactly with $V_h.$ Finally, the minimization problem in (\ref{2.11}) to be studied reduces to 

\begin{equation}
\label{4.10}
\Gl_{cl}(h)=\inf_{\BGf\in V_h} \frac{\int_{\Omega_h} \left\langle \BL_{0} e(\BGf), e(\BGf)\right\rangle d\Bx}
{E\cdot \|\mathrm{col}_3(\nabla\BGf)\|^{2}},
\end{equation}
which will be addressed in Section~\ref{sec:4.4}.



\subsection{Korn's Inequalities: Ansatz free lower bounds}

\begin{proof}[Proof of Theorem~{\ref{th:3.3}}]
\textbf{Part (i).} In what follows all the constants $C>0$ and $C_i>0$ will depend only on $L,\max k_\Gth,k,$ and $c.$ While we will do the proof for both spaces $V_h$ and $V_h^\Gth$ in parallel, we will provide additional details to address the more specific case $V=V_h^\Gth$ when needed. The tools and strategies developed in [\ref{bib:Gra.Har.1},\ref{bib:Gra.Har.4},\ref{bib:Harutyunyan.1}] will be adopted. Namely, we introduce the so-called simplified gradient $\BG,$ which is obtained by inserting $t=0$ in the denominators of the second column of (\ref{3.2}):
\begin{equation}
\label{4.11}
\BG=\begin{bmatrix} \phi_{t,t} & \phi_{t,\theta}-k_{\theta }\phi_\theta & \phi_{t,z}       \\[10pt]
\phi_{\theta,t} & \phi_{\theta,\theta}+k_{\theta }\phi_t& \phi_{\theta,z}      \\[10pt] 
\phi_{z,t} & \phi_{z,\theta}  & \phi_{z,z} \end{bmatrix}.
\end{equation}
For the sake of simplicity we will be proving all the estimates in Theorem~\ref{th:3.3} with the gradient $\nabla\BGf$ replaced by the simplified gradient $\BG$ first, and then we will make the reverse step by virtue of the obvious bounds 
 \begin{equation}
\label{4.12}
\|\BG^{sym}-e(\BGf)\|\leq \|\BG-\nabla\BGf\|\leq h\|\nabla \BGf\|\quad \text{for all}\quad \BGf\in H^1(\Omega_h),
\end{equation}
where $\BG^{sym}=\frac{1}{2}(\BG+\BG^T)$ is the symmetric part of $\BG.$ We aim to prove the following bounds 
\begin{equation}
\label{4.13}
\left\|\phi_{t, z}\right\|^{2} \leq \frac{C}{h}\|e(\BGf)\|^{2},\qquad
\left\|\phi_{\theta, z}\right\|^{2} \leq \frac{C}{\sqrt{h}}\|e(\BGf)\|^{2},\qquad
\left\|\phi_{z, z}\right\|^{2} \leq C\|e(\BGf)\|^{2},\quad\text{for all}\quad \BGf\in V.
\end{equation}
First of all note that by density we can without loss of generality assume that the field $\BGf$ is of class $C^2.$ For the $zz$ component we have the obvious estimate 
\begin{equation}
\label{4.14}
\|\phi_{z,z}\|\leq \|e(\BGf)\|.
\end{equation}
For the sake of simplicity, denote for any functions $f,g\in H^1(\Omega_h)$ the inner product $(f,g)=\int_{\Omega_h}fg.$ In order to estimate the 
second component of the third column, we can integrate by parts using the boundary conditions on the component $\phi_z$ in $z,$ and periodicity in $\Gth:$
\begin{align*}
\|\phi_{\Gth,z}\|^2&=(\phi_{\Gth,z},2\BG^{sym}_{\Gth z}-\phi_{z,\Gth})\\ 
&=2(\phi_{\Gth,z},\BG^{sym}_{\Gth z})-(\phi_{\Gth,z},\phi_{z,\Gth})\\
&=2(\phi_{\Gth,z},\BG^{sym}_{\Gth z})-(\phi_{\Gth,\Gth},\phi_{z,z})\\
&=2(\phi_{\Gth,z},\BG^{sym}_{\Gth z})-(\BG^{sym}_{\Gth\Gth}-k_\Gth \phi_{t},\BG^{sym}_{zz}),
\end{align*}
thus we obtain by the Cauchy-Schwartz inequality the estimate 
\begin{equation}
\label{4.15}
\|\phi_{\Gth,z}\|^2\leq 6\|\BG^{sym}\|(\|\BG^{sym}\|+\|k_\Gth\phi_t\|)\leq C\|\BG^{sym}\|(\|\BG^{sym}\|+\|\phi_t\|).
\end{equation}
Next we recall the following Korn and Korn interpolation inequalities proven in [\ref{bib:Gra.Har.4}, Thorems 3.1 and 3.2]:
\begin{equation}
\label{4.16}
\|\nabla \BGf\|^{2} \leq \frac{C}{h^{3/2}}\|e(\BGf)\|^2,\quad \|\nabla \BGf\|^{2} \leq C\|e(\BGf)\|\left(\frac{\left\|\phi_{t}\right\|}{h}+\|e(\BGf)\|\right),
\quad\text{for all}\quad \BGf\in V_h.
\end{equation}
Note that the inequalities in (\ref{4.16}) were derived in [\ref{bib:Gra.Har.4}] under the boundary conditions in $V_h,$ thus we will need to prove similar to (\ref{4.16}) estimates for the vector space $V_h^\Gth$ too. To that end, we invoke the following universal Korn interpolation inequality proven in [\ref{bib:Harutyunyan.1}, Theorem~3.1]. The estimate holds for any shells $\Omega_h$ with bounded principal curvatures and for any vector fields $\BGf\in H^1(\Omega_h)$ (even without any boundary conditions):
\begin{equation}
\label{4.17}
\|\nabla\BGf\|^2\leq C\left(\frac{1}{h}\|e(\BGf)\|\cdot\|\phi_t\|+\|\BGf\|^2+\|e(\BGf)\|^2\right),\qquad\text{for all}\qquad \BGf\in H^1(\Omega_h).
\end{equation} 
Observe that $z$ and $\Gth$ components of fields $\BGf\in V_h^\Gth,$ satisfy Poincar\'e inequality in the $z$ direction, thus keeping in mind (\ref{4.14}) and (\ref{4.15}), we obtain from (\ref{4.17}) the simplified estimate 
\begin{equation}
\label{4.18}
\|\nabla\BGf\|^2\leq C\left(\frac{1}{h}\|e(\BGf)\|\cdot\|\phi_t\|
+\|\phi_t\|^2+\|e(\BGf)\|^2\right),\qquad\text{for all}\qquad \BGf\in V.
\end{equation}
Note next that (\ref{4.18}) combined with (\ref{4.12}) and the Cauchy inequality implies the bound 
\begin{equation}
\label{4.19}
\|e(\BGf)\|\leq 2\|\BG^{sym}\|+ Ch^{1/2}\|\phi_t\|,
\end{equation}
thus we get from (\ref{4.18}) an analogous estimate 
\begin{equation}
\label{4.20}
\|\nabla\BGf\|^2\leq C\left(\frac{1}{h}\|\BG^{sym}\|\cdot\|\phi_t\|
+\frac{\|\phi_t\|^2}{h^{1/2}}+\|\BG^{sym}\|^2\right), \qquad\text{for all}\quad \BGf\in V.
\end{equation}
In order now to get a sharp estimate on the $z$ derivative of $\phi_t,$ we extend the cylinder $\Omega_h$ to the lower half-space by mirror reflection and accordingly the $t$ and $\Gth$ components of vector field $\BGf$ as even functions, and the $z$ component as an odd function, which is possible due to the imposed zero boundary conditions. The extended version are denoted by $\bar\Omega_h$ and $\bar \BGf,$ where we clearly have $\bar\BGf\in H^1(\bar\Omega_h).$ 
The point is that this extensions simply double all the norms under consideration, thus we can prove all the inequalities under consideration for the extended fields. 
The functions $\bar\phi_t,$ $\bar\phi_\Gth$ and $\bar\phi_z$ can then be written in Fourier space in the $z$ variable in $H^{1}:$ 
\begin{equation}
\label{4.21}
\begin{cases}
\bar \phi_t=\sum_{m=0}^\infty \bar \phi_t^m(t,\Gth)\cos(\frac{\pi mz}{L}),\\
\bar \phi_\Gth=\sum_{m=0}^\infty \bar \phi_\Gth^m(t,\Gth)\cos(\frac{\pi mz}{L}),\\
\bar \phi_z=\sum_{m=0}^\infty \bar \phi_z^m(t,\Gth)\sin(\frac{\pi mz}{L}).
\end{cases}
\end{equation}
Observe that in each norm under consideration the Fourier modes separate, thus we can prove all the inequalities under consideration for a fixed Fourier mode with a
wavenumber $m \geq 0.$ If $m=0,$ then we have $\phi_{t,z}^0=0$ and $\phi_{\Gth,z}^0=0,$ thus (\ref{4.13}) holds. Assuming $m>0,$ we have $\|\bar \phi_{\Gth,z}\|=m\|\bar\phi_{\Gth}\|,$
thus (\ref{4.15}) implies the bound 
\begin{equation}
\label{4.22}
m^2\|\bar\phi_{\Gth}\|^2\leq C\|\bar \BG^{sym}\|(\|\bar \BG^{sym}\|+\|\bar \phi_t\|).
\end{equation}
Next we have $e(\bar \BGf)_{\theta\theta}=\bar\BG_{\Gth\Gth}=\bar\phi_{\theta,\theta}+k(\theta)\bar\phi_t$, thus we have integrating by parts in $\Gth$ and utilizing periodicity: 
\begin{align*}
\int_{\Omega_h} k(\theta) \bar\phi_t^2&= \int_{\Omega_h} \bar\BG^{sym}_{\theta\theta} \bar\phi_t+\int_{\Omega_h}  (\bar\phi_\theta \bar\phi_t)_{,\theta}
+\int_{\Omega_h} \bar\phi_\theta \bar\phi_{t,\theta}\\
&=\int_{\Omega_h} \bar\BG^{sym}_{\theta\theta} \bar\phi_t+\int_{\Omega_h} \bar\phi_\theta \bar\phi_{t,\theta}\\
\end{align*}
We can substitute $\bar\phi_{t,\theta}=\bar\BG_{t\theta}+k(\theta)\bar\phi_\theta$ to get
\begin{equation}
\label{4.23}
\norm{\sqrt{k(\theta)}\bar\phi_t}^2\leq C\int_{\Omega_h} |\bar\phi_t \bar\BG^{sym}_{\theta\theta}|+|\bar\phi_\theta \bar\BG_{t\theta}|+|\bar\phi_\theta|^2.
\end{equation}
Observe that (\ref{4.22}), (\ref{4.23}) and an application of the Cauchy inequality imply the bound
\begin{equation}
\label{4.24}
\norm{\bar\phi_t}^4 \leq C\left(\norm{\bar\BG^{sym}}^4+\norm{\bar\phi_\theta}^2\norm{\bar\BG}^2\right).
\end{equation}
Note that (\ref{4.12}) implies that analogous to (\ref{4.20}) inequalities hold for $\nabla\BGf$ replaced by $\BG$, thus owing to (\ref{4.24}), (\ref{4.22}) and the new version off (\ref{4.20}), we discover 
\begin{equation}
\label{4.25}
\norm{\bar\phi_t}^2\leq C\left(\norm{\bar\BG^{sym}}^2+\frac{1}{m\sqrt{h}}\norm{\bar\phi_t}\norm{\bar\BG^{sym}}\right).
\end{equation}
It remains to note that, upon an application of the Cauchy inequality, (\ref{4.25}) implies the desired bound
\begin{equation}
\label{4.26}
\|\bar \phi_{t,z}\|^2=m^2\norm{\bar \phi_t}^2\leq \frac{\norm{\bar \BG^{sym}}^2}{h}.
\end{equation}
Observe that on one hand (\ref{4.26}) implies that $\norm{\bar \phi_t}\leq \frac{\norm{\bar \BG^{sym}}}{\sqrt{h}}$, thus we get from (\ref{4.20}) that
\begin{equation}
\label{4.27}
\|\nabla\BGf\|^2\leq \frac{C}{h\sqrt{h}}\|\BG^{sym}\|^2 \qquad\text{for all}\quad \BGf\in V.
\end{equation}
Next, owing back to (\ref{4.12}), we derive from (\ref{4.27}) the analogous estimate 
\begin{equation}
\label{4.28}
\|\nabla\BGf\|^2\leq \frac{C}{h\sqrt{h}}\|e(\BGf)\|^2 \qquad\text{for all}\quad \BGf\in V.
\end{equation}
On one hand, it remains to note that (\ref{4.28}) combined with (\ref{4.12}), (\ref{4.14}), (\ref{4.15}), and (\ref{4.26}) imply the bound 
\begin{equation}
\label{4.29}
\|\mathrm{col}_3(\nabla\BGf)\|^2\leq\frac{C}{h}\|e(\BGf)\|^2,\qquad\text{for all}\qquad \BGf\in V,
\end{equation}
which confirms one direction in (\ref{3.5}). The proof of the other direction is by an Ansatz construction and is postponed until Section~\ref{sec:4.4}. We now turn to the second part, where the curvature $k(\Gth)$ vanishes at finitely many points.\\
\textbf{Part (ii). } We will prove the Ansatz-free lower bounds
\begin{equation}
\label{4.30}
\|\mathrm{col}_3(\nabla\BGf)\|^2\leq\frac{C}{h^{8/5}}\|e(\BGf)\|^2,\qquad\|\nabla\BGf\|^2\leq\frac{C}{h^{12/7}}\|e(\BGf)\|^2,\qquad\text{for all}\qquad \BGf\in V.
\end{equation}
In the sequel we will be working with the extended vector field $\bar\BGf,$ but will drop the "bar" to keep the notation simpler. Assume first that the wavenumber $m$ is nonzero, i.e., $m\in\mathbb N.$ Observe that if $\norm{\phi_t}\leq \norm{e(\BGf)}$ then (\ref{4.18}) would imply (\ref{4.30}). We assume in the sequel that 
 \begin{equation}
 \label{4.31}
 \norm{e(\BGf)}\leq\norm{\phi_t}.
\end{equation}
We can assume without loss of generality that the domain of the $\Gth$ variable is $[-1,1].$ Also, for the simplicity of the presentation, we will assume that there is only one point on the curve $\BGa$ where the curvature vanishes, the general case being analogous. Consequently, assume $n=1$ and $\Gth_1=0$ in (\ref{3.6}). 
Let $\delta>0$ be a small parameter yet to be chosen and let $I_0= [-\delta,\delta], I_1=[-1,1]-I_0$, and $\BGf^i=\BGf\chi_{I_i}$, $i=0,1.$ 
Recall that (\ref{4.23}) implies 
\begin{equation}
\label{4.32}
\|\sqrt{k(\theta)}\phi_t\|^2\leq C\left( \norm{\phi_t}\cdot \norm{\BG^{sym}}+ \frac{\norm{\phi_\theta}^2}{\epsilon}+\epsilon\norm{\BG}^2\right),
\end{equation}
for any $\epsilon\in (0,\infty)$. From the obvious bound $k(\Gth)\leq C\sqrt{k(\Gth)}$ and inequality (\ref{4.15}) we have 
\begin{equation}
\label{4.33}
\norm{\phi_\theta}^2\leq \frac{C}{m^2}(\norm{\BG^{sym}}^2+\norm{\BG^{sym}}\cdot \|{\sqrt{k(\theta)}}\phi_t\|).
\end{equation}
Therefore combining (\ref{4.32}) and (\ref{4.33}) we arrive at
\begin{align*}
\|\sqrt{k(\theta)}\phi_t\|^2
&\leq C\left(\norm{\phi_t}\cdot \norm{\BG^{sym}}+\frac{\norm{\BG^{sym}}^2+\norm{\BG^{sym}}\cdot \|{\sqrt{k(\theta)}}\phi_t\|}{m^2\epsilon}
+\epsilon\norm{\BG)}^2\right),
\end{align*}
and applying Young's inequality again we arrive at the key estimate
\begin{equation}
\label{4.34}
\|\sqrt{k(\theta)}\phi_t\|^2
\leq C\left( \norm{\phi_t}\cdot \norm{\BG^{sym}}+\frac{\norm{\BG^{sym}}^2}{m^2\epsilon}+\frac{\norm{\BG^{sym}}^2}{m^4\epsilon^2}+\epsilon\|\BG\|^2\right).\\
\end{equation}
Next we utilize (\ref{4.34}) to bound $\norm{\phi_t^1}^2.$ We have by (\ref{3.6}) that $\min_{I_1} k(\theta)=k_\delta\geq c\delta^2,$ thus substituting $\epsilon=\delta^2 \eta^2$ we get from (\ref{4.34}) the bound 
\begin{align}
\label{4.35}
    \norm{\phi_t^1}^2&\leq \frac{1}{k_\delta}\|\sqrt{k(\theta)}\phi_t\|^2\\ \nonumber
 &\leq \frac{\|\phi_t\|^2}{50}+C\left(\left(\frac{1}{\delta^4}+\frac{1}{m^2\delta^4 \eta^2}+
 \frac{1}{m^4\delta^6\eta^4}\right)\norm{\BG^{sym}}^2+\eta^2\|\BG\|^2\right).   
 \end{align}
In order to bound $\norm{\phi_t^0}^2$ we choose a smooth cut-off function $\varphi\colon [-1,1]\to\mathbb R$ supported in $[-2\delta,2\delta]$ such that 
$$
\varphi(\theta)=\begin{cases} 1, &\theta\in I_0,\\
					0, &\theta\in [1,-1]-2I_0, \\
|\varphi'(\theta)|\leq\frac{2}{\delta}, & \theta\in [1,-1].	      
\end{cases}
$$
We have by the Poincare inequality that  
\begin{align}
\label{4.36}
\norm{\phi^0_t}^2 &\leq \norm{\phi_t \varphi}^2\\  \nonumber
&\leq \delta^2 \norm{\partial_\Gth(\phi_t \varphi)}^2\\  \nonumber
&\leq \delta^2\left(\norm{\BG \chi_ {2I_0}}^2+\norm{k(\theta) \phi_\theta \chi_ {2I_0}}^2+\frac{4}{\delta^2}\norm{\phi_t\chi_ {I_1}}^2\right)\\ \nonumber
&\leq C\delta^2\left(\norm{\BG}^2+\norm{\phi_\theta}^2\right)+4\norm{\phi^1_t}^2\\ \nonumber
&\leq C\delta^2 \norm{\BG}^2+4\norm{\phi^1_t}^2,
\end{align}
where we used the obvious bound $\|\phi_\Gth\|\leq C\|\BG\|.$
Putting now (\ref{4.35}) and (\ref{4.36}) together we arrive at 
\begin{equation}
\label{4.37}
\norm{\phi_t}^2\leq C (\delta^2+\eta^2)\norm{\BG}^2+C\left(\frac{1}{\delta^4}+\frac{1}{m^2\delta^4 \eta^2}+
 \frac{1}{m^4\delta^6\eta^4}\right)\norm{\BG^{sym}}^2
\end{equation}
Observe that the universal interpolation inequality in (\ref{4.18}) together with the bounds (\ref{4.14}), (\ref{4.22}), and (\ref{4.31}) imply another key estimate: 
\begin{equation}
\label{4.38}
\|\BG\|^2\leq\frac{C}{h}\|\BG^{sym}\|\|\phi_t\|.
\end{equation}
In order to estimate $\|\BG\|,$ we first combine (\ref{4.37}) and (\ref{4.38}) to get by an application of the Cauchy inequality:
\begin{equation}
\label{4.39}
\norm{\phi_t}^2\leq C \left(\frac{\delta^4+\eta^4}{h^2}+\frac{1}{\delta^4}+\frac{1}{m^2 \delta^4 \eta^2}+
 \frac{1}{m^4\delta^6\eta^4}\right)\norm{\BG^{sym}}^2.
\end{equation}
Finally keeping in mind that $m\geq 1,$ we choose $\eta=\delta=h^{1/7}$ to optimize (\ref{4.39}). This gives 
\begin{equation}
\label{4.40}
\norm{\BG}^2\leq \frac{C}{h^{12/7}}\norm{\BG^{sym}}^2,
\end{equation}
and consequently also the second inequality in (\ref{4.30}) through (\ref{4.12}). In order to prove the first inequality in (\ref{4.30}), we note that (\ref{4.38}) implies in particular the bound 
$$\|\phi_{t,z}\|^2\leq \frac{C}{h}\|\BG^{sym}\|\|\phi_t\|,$$
which is equivalent to
\begin{equation}
\label{4.41}
m^2\|\phi_t\|^2 \leq \frac{C}{m^2h^2}\norm{\BG^{sym}}^2.
\end{equation}
Next we choose $\eta=\delta$ in (\ref{4.39}) to get the simplified variant 
\begin{equation}
\label{4.42}
m^2\norm{\phi_t}^2\leq C \left(\frac{m^2\delta^4}{h^2}+\frac{m^2}{\delta^4}+\frac{1}{\delta^6}+
 \frac{1}{m^2\delta^{10}}\right)\norm{\BG^{sym}}^2.
\end{equation}
We need to obtain an optimal estimate for $m^2\norm{\phi_t}^2$ from (\ref{4.41}) and (\ref{4.42}) regardless of the value of $m,$ by choosing the parameter $\delta>0$ appropriately. We choose $\delta$ so that the values of the first and last summands on the right-hand side of (\ref{4.39}) coincide: 
$\frac{m^2\delta^4}{h^2}=\frac{1}{m^2\delta^{10}}$. This gives $\delta=\frac{h^{1/7}}{m^{2/7}}$ and (\ref{4.39}) reduces to 
\begin{equation}
\label{4.43}
m^2\norm{\phi_t}^2\leq C \left(\frac{m^{6/7}}{h^{10/7}}+\frac{m^{22/7}}{h^{4/7}}+\frac{m^{12/7}}{h^{6/7}}\right)\norm{\BG^{sym}}^2.
\end{equation} 
It remains to note that if $m\geq\frac{1}{ h^{1/5}},$ then (\ref{4.38}) would give 
\begin{equation}
\label{4.44}
m^2\|\phi_t\|^2 \leq \frac{C}{h^{8/5}}\norm{\BG^{sym}}^2.
\end{equation}
If otherwise $m\geq\frac{1}{ h^{1/5}},$ then we would get the same estimate (\ref{4.41}) this time from (\ref{4.40}) instead. Consequently (\ref{4.41}) is fulfilled independently of $m\in\mathbb N.$ Finally putting together (\ref{4.30}) and (\ref{4.41}) we arrive at the first estimate in (\ref{4.26}) in the case $m\geq 1.$ 
In the case $m=0$ there is no $z-$variable dependence, thus obviously have $\mathrm{col}_3(\nabla\BGf)=0$ and $\phi_\Gth=\phi_z=0,$ hence both lower bounds in 
(\ref{3.7}) become trivial. This completes the proof of the Ansatz-free lower bound parts of Theorem~\ref{th:3.3}.

\end{proof}

\subsection{The Ans\"atze}

\noindent \textbf{Part (i).} An Ansatz realizing the asymptotics in (\ref{3.5}) can be constructed in numerous ways. For instance, one Kirchhoff-like Ansatz which can be as such was constructed in [\ref{bib:Harutyunyan.2}], see also [\ref{bib:Harutyunyan.3}]. We present it here for the convenience of the reader. One chooses 
\begin{equation}
\label{4.45}
\begin{cases}
\phi_t=W(\frac{\Gth}{\sqrt{h}},\frac{z-L/2}{\sqrt{h}}),\\[10pt]
\phi_\Gth=-\frac{t}{\sqrt{h}} \cdot W_{,\Gth}(\frac{\Gth}{\sqrt{h}},\frac{z-L/2}{\sqrt{h}}),\\[10pt]
\phi_z=-\frac{t}{\sqrt{h}} \cdot W_{,z}(\frac{\Gth}{\sqrt{h}},\frac{z-L/2}{\sqrt{h}}),
\end{cases}
\end{equation}
where $W$ is a smooth function compactly supported in $(0,p)\times (0,L).$ Also, $W$ is chosen so that $W$ and all its first and second order derivatives be of order one. It is then easy to see that one gets for this choice $\|e(\BGf)\|\sim h$ and $\|\mathrm{col}_3(\nabla\BGf)\|=\sqrt{h}$ as $h\to 0.$ This gives the asymptotics 
$$\frac{\|e(\BGf)\|^2}{\|\mathrm{col}_3(\nabla\BGf)\|^2}\sim h,$$
as $h\to 0,$ i.e., (\ref{3.5}).\\
\noindent \textbf{Part (ii): Second estimate in (\ref{3.7}).} The idea and the novelty in this case is to localize the Ans\"atze at the zero curvature points, and make use of the fact that the curvature vanishes. Namely, assume again the domain of the variable $\Gth$ is $[-1,1]$ and the point $\Gth=0$ is a zero curvature point, i.e., 
$\Gth(0)=0$ and $c\Gth^2\leq k(\Gth)\leq \frac{1}{c}\Gth^2$ for all $\Gth\in [-1,1].$ Let $\delta=h^\alpha$ be a small parameter ($\alpha>0$) yet to be chosen. We adjust (\ref{4.45}) as 
\begin{equation}
\label{4.46}
\begin{cases}
\phi_t=W(\frac{\Gth}{\delta},\frac{z-L/2}{\delta}),\\[10pt]
\phi_\Gth=-\frac{t}{\delta} \cdot W_{,\Gth}(\frac{\Gth}{\delta},\frac{z-L/2}{\delta}),\\[10pt]
\phi_z=-\frac{t}{\delta}\cdot W_{,z}(\frac{\Gth}{\delta},\frac{z-L/2}{\delta}),
\end{cases}
\end{equation}
where $W(\theta,z)$ is again a smooth function, compactly supported on $D=(-1,1)^2$ such that $W$ and all its first and second order derivatives be of order one. 
Computing the simplified gradient $\BG$ in (\ref{4.11}) we get:
\begin{align*}
\BG=\begin{bmatrix} 
0     &     \frac{W_{,\theta}-k(\theta)tW_{,\theta}}{\delta }    &     \frac{W_{,z}}{\delta} \\[10pt]
-\frac{W_{,\theta}}{\delta}    &    \frac{tW_{,\theta\theta}}{\delta^2}+k(\theta)W      &      \frac{tW_{,\theta z}}{\delta^2}\\[10pt]
-\frac{W_{,z}}{\delta}    &     \frac{tW_{,z\theta}}{\delta^2}   &     -\frac{tW_{,zz}}{\delta^2}
\end{bmatrix}	
\end{align*}

and

$$
\BG^{sym}=\begin{bmatrix} 
0    &    -\frac{k(\theta)tW_{,\theta}}{2\delta}    &     0 \\[10pt]
-\frac{k(\theta)tW_{,\theta}}{2\delta}    &    \frac{tW_{,\theta\theta}}{\delta^2}+k(\theta)W  &   \frac{2tW_{,\theta z}}{\delta^2} \\[10pt]
0    &    \frac{2tW_{,\theta z}}{\delta^2}    &    -\frac{tW_{,zz}}{\delta^2}
\end{bmatrix}.
$$
Now choosing $\delta=h^{1/4}$ and recalling that $k(\Gth)\sim\Gth^2,$ it is easy to see that $\|\mathrm{col}_3(\nabla\BGf)\|^2=h,$ 
and $\|e(\BGf)\|^2\sim h^{5/2},$ as $h\to 0.$ This realizes the asymptotics 
$$\frac{\|e(\BGf)\|^2}{\|\mathrm{col}_3(\nabla\BGf)\|^2}\sim h^{3/2}$$ 
as $h\to 0,$ i.e., the right-hand side of the second inequality in (\ref{3.7}). 

\begin{remark}
\label{rem:4.2}
It is easy to see that if the curvature $k(\theta)$ has a zero of order $\beta\geq 2,$ then by choosing $\delta=h^{\frac{1}{\beta+2}},$ the Ansatz in (\ref{4.43}) would actually give us $$\frac{\|e(\BGf)\|^2}{\|\mathrm{col}_3(\nabla\BGf)\|^2}\sim h^{\frac{2\beta+2}{\beta+2}},$$
as $h\to 0.$
\end{remark}

\noindent \textbf{Part (iii): First estimate in (\ref{3.7}).} We construct the Ansatz utilizing the idea of linearization in $t$ suggested in [\ref{bib:Gra.Har.1}]. Namely we seek the Ansatz in the following form 
\begin{equation}
\label{4.47}
\begin{cases}
\phi_t=u,\\
\phi_\theta=tv_1+v_2,\\
\phi_z=tw_1+w_2,
\end{cases}
\end{equation}
where the functions $u,v_1,v_2,w_1$ and $w_2$ depend only on $\Gth$ and $z.$ The implified gradient will then be given by 
\begin{equation*}
\BG=
\begin{bmatrix} 
0       &       u_{,\Gth}-k(\theta)\left(tv_1+v_2\right)     &      u_{,z} \\[10pt]
v_1    &        tv_{1,\theta}+v_{2,\Gth}+k(\Gth)u           &        tv_{1,z}+v_{2,z}\\[10pt]
w_1    &      tw_{1,\theta}+w_{2,\theta}                        &          tw_{1,z}+w_{2,z}
\end{bmatrix}.
\end{equation*}
In order to make the symmetric part of the gradient small, we choose the functions $u,v_1,v_2,w_1,$ and $w_2$ to satisfy the relationships
\begin{equation}
\label{4.48}
v_1=-u_{,\theta},\qquad w_1=-u_{,z},\qquad v_{2,\theta}=-k(\theta)u.
\end{equation}
This will reduce the simplified gradient to 
\begin{equation}
\label{4.49}
\BG=
\begin{bmatrix} 
0       &       u_{,\Gth}-k(\theta)\left(tv_1+v_2\right)     &      u_{,z} \\[10pt]
v_1    &        tv_{1,\theta}         &        tv_{1,z}+v_{2,z}\\[10pt]
w_1    &      tw_{1,\theta}+w_{2,\theta}                        &          tw_{1,z}+w_{2,z}
\end{bmatrix},
\end{equation}
and the symmetric part will be
\begin{equation}
\label{4.50}
\BG^{sym}=
\begin{bmatrix} 
0       &      -\frac{1}{2}k(\theta)\left(tv_1+v_2\right)    &    0 \\[10pt]
-\frac{1}{2}k(\theta)\left(tv_1+v_2\right)    &     tv_{1,\theta}    &      \frac{1}{2}( t(v_{1,z}+w_{1,\theta})+v_{2,z}+w_{2,\theta})\\[10pt]
0      &    \frac{1}{2}( t(v_{1,z}+w_{1,\theta})+v_{2,z}+w_{2,\theta})    &      tw_{1,z}+w_{2,z}
\end{bmatrix} 
\end{equation}
Thus to make the $\Gth z$ component small, we need a new relationship:
\begin{equation}
\label{4.51}
w_{2,\theta}=-v_{2,z},
\end{equation}
which simplifies $\BG^{sym}$ further to
\begin{equation}
\label{4.52}
e(\BG)=
\begin{bmatrix} 
0       &      -\frac{1}{2}k(\theta)\left(tv_1+v_2\right)    &    0 \\[10pt]
-\frac{1}{2}k(\theta)\left(tv_1+v_2\right)    &     tv_{1,\theta}    &      \frac{t}{2}(v_{1,z}+w_{1,\theta})\\[10pt]
0      &    \frac{t}{2}(v_{1,z}+w_{1,\theta})    &      tw_{1,z}+w_{2,z}
\end{bmatrix}.
\end{equation}
Let now $W(\theta,z)$ be a smooth compactly supported function on $D=(-1,1)\times(0,L)$ such that $W$ and all its first, second, and third order derivatives be of order one. 
We choose 
\begin{equation}
\label{4.53}
\begin{cases}
w_2=-k^2(\theta)W_{,z}(\frac{\theta}{\delta},z)\\[5pt]
v_2=2k(\theta)k'(\theta)W(\frac{\theta}{\delta},z)+\frac{k(\theta)^2}{\delta}W_{,\theta}(\frac{\theta}{\delta},z)\\[5pt]
u=-\frac{2k'^2(\theta)}{k(\theta)}W(\frac{\theta}{\delta},z)-2k''(\theta)W(\frac{\theta}{\delta},z)-\frac{4k'(\theta)}{\delta}W_{,\theta}(\frac{\theta}{\delta},z)-\frac{k(\theta)}{\delta^2}W_{,\theta\theta}(\frac{\theta}{\delta},z)\\[5pt]
v_1=-u_{,\theta}\\[5pt]
w_1=-u_{,z},
\end{cases}
\end{equation}
where $\delta=h^\beta$ ($\beta>0$) is a small parameter yet to be chosen. From the fact that $\BGa$ is of class $C^5$ and $k(0)=k'(0)=0,$ $k''(0)>0,$ we have for small enough $h$ the obvious bounds $|k'''(\theta)|+|k''(\theta)|\leq C,$ $|k'(\theta)|\leq C|\Gth|,$ $c\Gth^2\leq |k(\theta)|\leq \frac{1}{c}\Gth^2,$ for $\Gth\in (-\delta,\delta).$
Consequently we can easily verify that
$$
\frac{\norm{\BG^{sym}}^2}{\norm{\BG}^2}\sim \max\{h^{6\beta}(h^{2-2\beta}+h^{6\beta}),h^{2-2\beta},h^2+h^{10\beta}\}
$$
as $h\to 0.$ In order to minimize the left hand side we choose $\beta=1/6$ that gives the desired result
$$\frac{\norm{e(\BGf)}^2}{\norm{\nabla\BGf}^2}\sim h^{\frac{5}{3}}.$$




\subsection{The buckling load}
\label{sec:4.4}

\begin{proof}[Proof of Theorem~\ref{th:3.1}] 
In this section we will study the stability of the homogeneous trivial branch given by (in (\ref{4.1}))
\begin{equation}
\label{4.54}
y_t=(1+a(\lambda))(t+\bm\alpha(\theta)\cdot \bm e_t(\theta)), \qquad y_\theta=(1+a(\lambda))\bm\alpha(\theta)\cdot \bm e_\theta(\theta), \qquad y_z=(1-\lambda)z, 
\end{equation}
within the theory of buckling of slender structures presented in Section~\ref{sec:2}. Let us start by verifying that the conditions for the applicability of the theory are fulfilled. These consist of the conditions in Definition~\ref{def:2.1} and Theorem~\ref{th:2.5}. The conditions in Definition~\ref{def:2.1} have already been verified by 
Lemma~\ref{4.1}, while the conditions in Theorem~\ref{th:2.5} immediately follow from Theorem~\ref{th:3.3} and Remark~\ref{rem:3.4}. This implies that in fact we can 
calculate the buckling load as in (\ref{2.11}), which we have proven reduces to the minimization problem in (\ref{4.10}), i.e., 
\begin{equation}
\label{4.55}
\Gl(h)=\inf_{\BGf\in V} \frac{\int_{\Omega_h} \left\langle \BL_{0} e(\BGf), e(\BGf)\right\rangle d\Bx}{E\cdot \|\mathrm{col}_3(\nabla\BGf)\|^{2}}.
\end{equation}
By properties (P4) and (P5) the tensor $\BL_0$ is positive definite on symmetric matrices, i.e., there exist a positive constant $C>0$ such that 
$$\frac{1}{C}\|e(\BGf)\|^2\leq \int_{\Omega_h} \left\langle \BL_{0} e(\BGf), e(\BGf)\right\rangle d\Bx \leq C\|e(\BGf)\|^2,$$
for all $\BGf\in H^1(\Omega_h).$ Hence (\ref{4.52}) yields the asymptotics 
\begin{equation}
\label{4.56}
\Gl(h)\sim \inf_{\BGf\in V} \frac{\|e(\BGf)\|^2}{\|\mathrm{col}_3(\nabla\BGf)\|^{2}}, 
\end{equation}
and Theorem~\ref{3.1} immediately follows from Theorem~\ref{th:3.3}. 

\end{proof}

\section*{Acknowledgement}
We would like to thank Yury Grabovsky for helpful comments. Anonymous referees are also thanked for useful suggestions on how to improve the presentation of the manuscript. This material is supported by the National Science Foundation under Grants No. DMS-1814361.

\appendix

\section{Explicit trivial branch for Neo-Hookean materials} 
\setcounter{equation}{0}

In this section we provide an explicit form for the function $a(\lambda)$ in (\ref{4.1}) and (\ref{4.8}) for the special case of Neo-Hookean solids. 
In that case the Piola-Kirchhoff stress tensor is given by  
$$\BP(\BF)=\tilde \mu(\BF-\BF^{-T})+ 2\tilde{\lambda}(J-1)J\BF^{-T},\quad J=\det(\BF),$$
where $\tilde \mu$ and $\tilde\Gl$ are the Lam'e parameters\footnote{Note that we are using the symbol "tilde" over the letters to avoid confusion between the Lam\'e parameter $\tilde\Gl$ and the loading parameter $\Gl.$}. Letting $b(\Gl)=(1+a(\lambda))^2,$ the system (\ref{4.3}) will reduce to 
$$b(\Gl)^2(1-\lambda)^2-b(\Gl)(1-\lambda-\frac{\tilde\mu}{\tilde{\lambda}})-\frac{\tilde\mu}{\tilde{\lambda}}=0,$$
which gives the solution
$$b(\Gl)=\frac{1-\lambda-\frac{\tilde\mu}{\tilde{\lambda}}+\sqrt{(1-\lambda-\frac{\tilde\mu}{\tilde{\lambda}})^2+4(1-\lambda)^2\frac{\tilde\mu}{\tilde{\lambda}}}
}{2(1-\lambda)^2},$$
and for $a$ we get
$$a(\lambda)=\frac{\sqrt{1-\lambda-\frac{\tilde\mu}{\tilde{\lambda}}+\sqrt{(1-\lambda-\frac{\tilde\mu}{\tilde{\lambda}})^2
+4(1-\lambda)^2\frac{\tilde\mu}{\tilde{\lambda}}}}}{\sqrt{2}(1-\lambda)}-1.
$$


\end{document}

%% file: mymacros.tex
\usepackage{amssymb,bm}
\usepackage{amsmath}
\usepackage{amsthm}
\usepackage{graphicx}
\usepackage[active]{srcltx} 
\usepackage{hyperref}
\hypersetup{pdfborder=0 0 0}

\newtheorem{theorem}{{\sc Theorem}}[section]

\newtheorem{lemma}[theorem]{{\sc Lemma}}

\newtheorem{remark}[theorem]{Remark}

\newtheorem{definition}[theorem]{Definition}

\def\XXint#1#2#3{{\setbox0=\hbox{$#1{#2#3}{\int}$ }
\vcenter{\hbox{$#2#3$ }}\kern-.6\wd0}}

\newcommand{\Gl}{\lambda}

\newcommand{\Gth}{\theta}

\newcommand{\Gs}{\sigma}

\bmdefine\BGa{\alpha}
\bmdefine\BGb{\beta}
\bmdefine\BGd{\delta}
\bmdefine\BGe{\epsilon}
\bmdefine\BGve{\varepsilon}
\bmdefine\BGf{\phi}
\bmdefine\BGvf{\varphi}
\bmdefine\BGg{\gamma}
\bmdefine\BGc{\chi}
\bmdefine\BGi{\iota}
\bmdefine\BGk{\kappa}
\bmdefine\BGl{\lambda}
\bmdefine\BGn{\eta}
\bmdefine\BGm{\mu}
\bmdefine\BGv{\nu}
\bmdefine\BGp{\pi}
\bmdefine\BGth{\theta}
\bmdefine\BGvth{\vartheta}
\bmdefine\BGr{\rho}
\bmdefine\BGvr{\varrho}
\bmdefine\BGs{\sigma}
\bmdefine\BGvs{\varsigma}
\bmdefine\BGt{\tau}
\bmdefine\BGj{\tau}
\bmdefine\BGu{\upsilon}
\bmdefine\BGo{\omega}
\bmdefine\BGx{\xi}
\bmdefine\BGy{\psi}
\bmdefine\BGz{\zeta}
\bmdefine\BGD{\Delta}
\bmdefine\BGF{\Phi}
\bmdefine\BGG{\Gamma}
\bmdefine\BGL{\Lambda}
\bmdefine\BGP{\Pi}
\bmdefine\BGT{\Theta}
\bmdefine\BGS{\Sigma}
\bmdefine\BGU{\Upsilon}
\bmdefine\BGO{\Omega}
\bmdefine\BGX{\Xi}
\bmdefine\BGY{\Psi}

\bmdefine\BCA{{\mathcal A}}
\bmdefine\BCB{{\mathcal B}}
\bmdefine\BCC{{\mathcal C}}
\bmdefine\BCD{{\mathcal D}}
\bmdefine\BCE{{\mathcal E}}
\bmdefine\BCF{{\mathcal F}}
\bmdefine\BCG{{\mathcal G}}
\bmdefine\BCH{{\mathcal H}}
\bmdefine\BCI{{\mathcal I}}
\bmdefine\BCJ{{\mathcal J}}
\bmdefine\BCK{{\mathcal K}}
\bmdefine\BCL{{\mathcal L}}
\bmdefine\BCM{{\mathcal M}}
\bmdefine\BCN{{\mathcal N}}
\bmdefine\BCO{{\mathcal O}}
\bmdefine\BCP{{\mathcal P}}
\bmdefine\BCQ{{\mathcal Q}}
\bmdefine\BCR{{\mathcal R}}
\bmdefine\BCS{{\mathcal S}}
\bmdefine\BCT{{\mathcal T}}
\bmdefine\BCU{{\mathcal U}}
\bmdefine\BCV{{\mathcal V}}
\bmdefine\BCW{{\mathcal W}}
\bmdefine\BCX{{\mathcal X}}
\bmdefine\BCY{{\mathcal Y}}
\bmdefine\BCZ{{\mathcal Z}}

\bmdefine\Bzr{ 0}
\bmdefine\Ba{ a}
\bmdefine\Bb{ b}
\bmdefine\Bc{ c}
\bmdefine\Bd{ d}
\bmdefine\Be{ e}
\bmdefine\Bf{ f}
\bmdefine\Bg{ g}
\bmdefine\Bh{ h}
\bmdefine\Bi{ i}
\bmdefine\Bj{ j}
\bmdefine\Bk{ k}
\bmdefine\Bl{ l}
\bmdefine\Bm{ m}
\bmdefine\Bn{ n}
\bmdefine\Bo{ o}
\bmdefine\Bp{ p}
\bmdefine\Bq{ q}
\bmdefine\Br{ r}
\bmdefine\Bs{ s}
\bmdefine\Bt{ t}
\bmdefine\Bu{ u}
\bmdefine\Bv{ v}
\bmdefine\Bw{ w}
\bmdefine\Bx{ x}
\bmdefine\By{ y}
\bmdefine\Bz{ z}
\bmdefine\BA{ A}
\bmdefine\BB{ B}
\bmdefine\BC{ C}
\bmdefine\BD{ D}
\bmdefine\BE{ E}
\bmdefine\BF{ F}
\bmdefine\BG{ G}
\bmdefine\BH{ H}
\bmdefine\BI{ I}
\bmdefine\BJ{ J}
\bmdefine\BK{ K}
\bmdefine\BL{ L}
\bmdefine\BM{ M}
\bmdefine\BN{ N}
\bmdefine\BO{ O}
\bmdefine\BP{ P}
\bmdefine\BQ{ Q}
\bmdefine\BR{ R}
\bmdefine\BS{ S}
\bmdefine\BT{ T}
\bmdefine\BU{ U}
\bmdefine\BV{ V}
\bmdefine\BW{ W}
\bmdefine\BX{ X}
\bmdefine\BY{ Y}
\bmdefine\BZ{ Z}

